\documentclass[twoside]{imsart}

\usepackage{amssymb}
\usepackage{amsmath}
\usepackage{amsthm}
\usepackage{graphicx}
\usepackage{mathabx}
\RequirePackage{multirow}

\usepackage{geometry}                
\usepackage{color}

\newtheorem{conj}{\sc Conjecture}
\newtheorem{thm}[conj]{\sc Theorem}
\newtheorem{cor}[conj]{\sc Corollary}
\newtheorem{prop}[conj]{\sc Proposition}
\newtheorem{lemma}[conj]{\sc Lemma}

\def\to{\rightarrow}

\def\EE{ {\rm I} \kern-.15em {\rm E} }
\def\PP{ {\rm I} \kern-.15em {\rm P} }

\def\wh{\widehat}
\def\eps{\varepsilon}
\def\mds{\medskip}

\def\F{\mbox{$\mathcal F$}}
\def\G{\mbox{$\mathcal G$}}
\def\B{\mbox{$\mathcal B$}}
\def\C{\mbox{$\mathcal C$}}

\def\R{\mbox{$\mathcal R$}}
\def\T{\mbox{$\mathcal T$}}
\def\H{\mbox{$\rm H$}}
\def\MM{\mathbb M}

\def\TT{\mathbb T}
\def\FF{\mathbb F}

\def\CC{\mathbb C}
\def\ZZ{\mathbb Z}
\def\HH{\mathbb H}
\def\RR{\mathbb R}
\def\EE{\mathbb{E}}
\def\PP{\mathbb{P}}

\def\YY{\mathbb{Y}}

\def\i{\mathbf{i}}
\def\u{\mathbf{u}}
\def\v{\mathbf{v}}
\def\x{\mathbf{x}}

\def\s{\mathbf{s}}

\def\X{\mathbf{X}}

\def\1{\mathbf{1}}

\begin{document}

{\huge  Asymptotic Total Variation Tests for Copulas}

\bigskip

{\sc JEAN-DAVID FERMANIAN, DRAGAN RADULOVI\'C and MARTEN WEGKAMP\thanksref{t2}}

\medskip


\thankstext{t2}{The research of Wegkamp is supported in part by NSF DMS 1007444 and NSF DMS 1310119 grants.}

\runtitle{ATV Tests for Copulas}
\runauthor{J.-D. Fermanian, D. Radulovi\'c and M. Wegkamp}


{\it
Jean-David Fermanian\\
Crest-Ensae\\ 3 av. Pierre Larousse\\ 92245 Malakoff cedex, France\\ E-mail: jean-david.fermanian@ensae.fr

\medskip

Dragan Radulovi\'c\\
Department of Mathematics\\ Florida Atlantic University, USA\\ E-mail: radulovi@fau.edu

\medskip

Marten Wegkamp\\ Department of Mathematics \& Department of Statistical Science\\ Cornell University, USA\\ E-mail: marten.wegkamp@cornell.edu
}

\medskip

\noindent

We propose a new platform of  goodness-of-fit tests  for copulas, based on
empirical copula processes and  nonparametric bootstrap
counterparts. The standard Kolmogorov-Smirnov type test for copulas
that takes the supremum of the empirical copula process indexed by
orthants is extended by test statistics based on  the
empirical copula process indexed by families of $L_n$ disjoint
boxes, with $L_n$ slowly tending to infinity. Although the
underlying empirical process does not converge, the critical values of our
new test statistics can be consistently estimated by nonparametric
bootstrap techniques, under simple or composite null assumptions.
We implemented a particular example of these tests and our simulations  confirm that the power of the new procedure is oftentimes
higher than the power of the standard Kolmogorov-Smirnov or the
Cram\'er-von Mises tests for copulas.

\medskip

{\it AMS 2000 subject classifications:} Primary 60F17, 60K35; secondary 60G99.

{\it Keywords:} Bootstrap, copula, empirical copula process, goodness-of-fit Test, weak Convergence.

\maketitle

\section{Introduction}
\setcounter{equation}{0}

This paper introduces  new powerful goodness-of-fit (GOF) tests for
copulas in $[0,1]^d$, $d\geq 2$, based on the empirical copula process
 \begin{eqnarray}
 \label{Z_n}
\ZZ_n(\u) = \sqrt{n}(\CC_n -C)(\u), \qquad  \u=(u_1,\ldots,u_d)
\in [0,1]^d,\end{eqnarray} given a sample of $n$ independent random
vectors $\X_i=(X_{i1},\ldots,X_{id})\in \RR^d$, $i=1,\ldots,n$,  from a common distribution function $H$. Let $C$ be the
 associated copula function, as given by Sklar's Theorem (Sklar, 1959).
 Here $\CC_n$ is the
usual empirical copula, as introduced by Deheuvels (1979): denoting
by $\HH_n$ the joint cdf of the sample $(\X_1,\ldots,\X_n)$,
$\FF_{n,j}$ the $j$-th empirical cdf associated to
$(X_{1j},\ldots,X_{nj})$, $j=1,\ldots,d$, and $\FF_{n,j}^{-}$ its
empirical quantile function, we have
$$\CC_n(\u)= \HH_n(\FF_{n,1}^{-}(u_1),\ldots,\FF_{n,d}^{-}(u_d))$$ by definition, for every $\u=(u_1,\ldots,u_d)\in [0,1]^d$.
The Kolmogorov-Smirnov (KS) test statistic for testing of the null
hypothesis  $\H_0:\, C=C_0$ is
\begin{eqnarray}
\label{KS}
{\rm KS}_n = \sup _{\u \in [0,1]^d} | \sqrt{n}(\CC_n -C_0)(\u) |.
\end{eqnarray}
The Cram\'er-von Mises statistic (CvM) is
\begin{eqnarray}
\label{CvM}
{\rm  CM}_n = \int  \{  \sqrt{n}(\CC_n -C_0)(\u) \}^2 \, {\rm d}\CC_n(\u).
\end{eqnarray}

\medskip



It is well-known, see, for instance, Fermanian et al. (2004),
that $\ZZ_n$ and its bootstrap counterpart $\ZZ_n^*$, defined in
(\ref{Z_n^*}) below, both converge weakly to the same tight
Gaussian process in $\ell^\infty([0,1]^d)$ under the null
hypothesis. Therefore, we can compute the $\alpha$-upper points of
${\rm KS}_n$ and ${\rm CM}_n$ via the bootstrap.
To the best of our knowledge, all the
proposed GOF tests rely on simulation-based procedures to
calculate their corresponding p-values, with the notable
exception of the distribution-free test statistics of Fermanian (2005). The latter idea has been further developed  by Scaillet (2007) and Fermanian and Wegkamp (2012). A parametric bootstrap has been proposed (Genest and
R\'emillard, 2008) to tackle composite null hypotheses, while
R\'emillard and Scaillet (2009) advocate the use of the
multiplier central limit theorem to build an alternative bootstrap
empirical copula process.  B\"ucher and Dette (2010) give a survey and a comparison of various bootstrap methods.\\

The goal of this paper is to develop   more powerful tests than 
 the KS test (\ref{KS}) and CvM test (\ref{CvM}) for simple and composite null hypotheses.  The next section offers a class of such tests. For instance,  
 in the case of a null simple hypothesis $\H_0:\, C=C_0$, we propose the  test that rejects
$\H_0$ for large values of the test statistic
\begin{equation}
\label{T_n:B}
\mathbb{T}_{n}:=\sup_{B_{1},\ldots
,B_{L_n}}\sum_{k=1}^{L_n}|\mathbb{Z}_{n}(B_{k})|.
\end{equation}
The supremum is taken over all {disjoint}
 boxes
$B_{1},\ldots ,B_{L_n}\subset \lbrack 0,1]^{d}$ of the form $\prod_{j=1}^d (a_j,b_j]$, using the convention
\begin{equation}
\mathbb{Z}_{n}((a_1,b_1]\times\cdots\times (a_d,b_d])=\Delta_{a_1,b_1}^1 \Delta_{a_2,b_2}^2 \cdots \Delta_{a_d,b_d}^d \ZZ_n(\u),
\label{Z_n:box}
\end{equation}for any arbitrary point $\u \in [0,1]^d$ and
for all $0\leq a_j< b_j\leq 1$, $j=1,\ldots,d$. Here, we have used the usual operators $\Delta^j$ defined for every function $f$ by
$$\left(\Delta^j_{a,b}f\right)(\u)=f(u_1,\ldots,u_{j-1},b,u_{j+1},\ldots,u_d) - f(u_1,\ldots,u_{j-1},a,u_{j+1},\ldots,u_d),  $$
for all $\u\in[0,1]^d$, and all real numbers $a$ and $b$.\\

We will also consider the related  statistics  
\begin{equation}
\widetilde{\TT}_{n}=\max_{B_{1},...,B_{L_{n}}}\sum_{i=1}^{L_{n}}|%
\ZZ_{n}(B_{i})|\text{ ,\ }
\label{tildeTT}
\end{equation}
with the maximum  taken over all disjoint rectangles $%
B_{1},...,B_{L_{n}}$ of the form $B=\Pi _{j=1}^{d}(a_{j},b_{j}]$ with $%
a_{j},b_{j}$ belonging to a grid $\{ {n^{-1/d}}, {2}{n^{-1/d}},\ldots, {\lfloor n^{1/d} \rfloor}{n^{-1/d}}\}.$
Asymptotically, $\widetilde{\TT}_{n}$ and $\TT_{n}$
are the same (since $|\widetilde{\TT}_{n}-\TT_{n}|=o_{p}(1)$ by Lemma 9 and Proposition 10 in Section 5), but $\widetilde{\TT}_{n}$ is computationally much more
tractable. \\

Now, if   $L_n=L$ for all $n$,   the collection of boxes is
sufficiently small that we can still appeal to the weak convergence
of $\ZZ_n$ and $\ZZ_n^*$  in conjunction with the continuous mapping
theorem, to obtain $\alpha$-upper points of the test statistic
$\TT_n$ via the bootstrap. Taking $L_n=+\infty$ for all $n$, or
equivalently, if we consider  all families of disjoint boxes in
$[0,1]^d$ (possibly partitions), the statistic $\TT_n$ is equal to
the total variation distance $TV(\ZZ_n)$  of $\ZZ_n$.
 The resulting test is not statistically  meaningful as $TV(\ZZ_n)$ is maximal, to wit, $TV(\ZZ_n)=n^{1/2} \to +\infty$.
The problem is to find a rich collection that quickly detects
departure from the null, but still yields a consistent test. The
main novelty of our approach is the fact that we let  $L_n$, the
number of boxes,  slowly tend to $ \infty $  in that $L_n \sim (\log
n)^\gamma$, $0<\gamma<1$. While in this case the process  $\ZZ_n$
no longer converges,  Theorem \ref{Thm1} in Section 2 states that we
can still consistently estimate the distribution of the process
$\ZZ_n$ by the bootstrap.
We refer to our procedure as the  Asymptotic  Total Variation (ATV) test. The considered families of boxes are finer and finer, presumably improving the power of the test, while for each $n$ large enough, we still have a consistent test in that we control the type 1 error. A key observation is that under the null hypothesis ${\rm H}_0:\,C=C_0$, we have   $\TT_n\le L_n \sup_{B} |\ZZ_n(B)| =O_p(L_n)$, while under the alternative ${\rm H}_A:\,C=C_1$ for some fixed $C_1\ne C_0$, $\TT_n$ is much larger since the bias is  at least of order $O({n^{1/2}})$. \\

Theorem  \ref{Thm1} extends the surprising result  obtained by
Radulovi\'c (2012) for empirical processes indexed by sums of
indicator functions of  VC-graph classes (see
Theorem~\ref{dragan2012} in the appendix). We require very mild
conditions on the copula function $C$. This is one of the few
notable exceptions known to us in the literature where the bootstrap
``works'', that is, the conditional bootstrap distribution
consistently estimates the distribution of the test statistic, while
the distribution of the statistic  itself does not converge. For
other instances of this phenomenon, we refer to Bickel and Freedman
(1983) and, more recently, Radulovi\'c (1998, 2012, 2013).

Section 3 considers the more general hypothesis that the underlying copula $C$  belongs to some parametric copula family $\{C_\theta,\, \theta\in\Theta\subset \RR^p\}$.
Given a sufficiently regular estimator $\wh\theta$ and its bootstrap counterpart $\wh\theta^*$,
we adjust our statistic (\ref{T_n:B}) and its {non-parametric} bootstrap counterpart to obtain a consistent level $\alpha$ test (Theorem \ref{CoreThParam}).
Again, the result is established under very mild regularity  conditions on the copula $C_\theta$ and the estimators $\wh\theta$ and $\wh\theta^*$. Incidentally, we introduce a new bootstrap procedure under composite null hypotheses, an alternative to the usual parametric bootstrap or the multiplier CLT.

Section 4 then reports a small numerical study where we show that,
in complex but realistic situations, our test (\ref{T_n:B}) is
superior to the Kolmogorov-Smirnov  and the Cram\'er-von Mises
tests. We also comment on a possible inadequacy in the way the
copula GOF tests are commonly evaluated. Finally,  the proofs are
collected in Section~\ref{myproofs}. The appendix contains some
technical results from Segers (2012) and Radulovi\'c (2012) and a description of the implementation of the proposed tests.\\

\section{The Asymptotic Total Variation Test}
\label{GKST}
\setcounter{equation}{0}

{\sc Notations.} Let $H$ be the distribution function of the random
vector $\X$ with marginals $F_1,\ldots,F_d$. We will assume
throughout the paper that $H$ is continuous. Let
$(\X_1,\ldots,\X_n)$ be independent copies of $\X$. We denote the
generalized inverse of a distribution function $F$ by $F^{-}$. For instance, $ F_j^{-}( u) =
\inf\{ x\, |\, F_j(x) \ge u \}$.
The empirical counterparts of $H$ and any $F_j$ are, respectively,
\begin{eqnarray*} \HH_n(\x)&=&\frac{1}{n} \sum_{i=1}^n \1\{ \X_i\le \x \},\ \x\in\RR^d \\
 \FF_{n,j}(x) &=& \frac{1}{n} \sum_{i=1}^n \1\{ X_{i,j}\le x\}, \ x\in\RR,\ \; j=1,\ldots,d.
 \end{eqnarray*}
The copula function of $\X$ is
$ C(\u)= H(F_1^{-}(u_1),\ldots,F_d^{-}(u_d))$, $\u=(u_1,\ldots,u_d)\in [0,1]^d$,
and its empirical estimate is
$ \CC_n(\u)= \HH_n(\FF_{n,1}^{-}(u_1),\ldots, \FF_{n,d}^{-}(u_d))$.
The empirical copula process $\mathbb{Z}_{n}(\u)=\sqrt{n}(\mathbb{C}
_{n}-C)(\u)$ is already defined in (\ref{Z_n}).
We define $\mbox{$\mathcal F_n$}$ as the class of functions
\begin{equation}
\label{F:class}
f(\x)=\sum_{k=1}^{L_n}c_{k}\1\{\x \in B_{k}\},
\end{equation}
with $c_{k}\in \{-1,+1\}$ and disjoint boxes $B_{k}$ of the form
$\prod_{j=1}^d (a_j,b_j]$ in the unit cube $[0,1]^{d}$, for all
$1\leq k\leq L_n$. We let
\[
\mathbb{Z}_{n}(f)=\sum_{k=1}^{L_n}c_{k}\mathbb{Z}_{n}(B_{k}),
\]
and observe that
\[
\mathbb{T}_{n}=\sup_{f\in \mbox{$\mathcal F_n$}}|\mathbb{Z}_{n}(f)|=
\sup_{B_{1},\ldots ,B_{L_n}}\sum_{k=1}^{L_n}|\mathbb{Z}_{n}(B_{k})|,
\]
where the supremum is taken over all disjoint boxes $B_{1},\ldots
,B_{L_n} $ of the unit square $[0,1]^{d}$.

If $L_n=L$ for all $n$, then $\F_n=\F$ and   $\ZZ_n$  converges in $\ell^\infty(\F)$ to a
Gaussian process under regularity conditions on $C$, see, for
instance,  Fermanian et al. (2004) and Segers (2012). As a
consequence of the continuous mapping theorem, $\TT_n$ trivially
converges weakly as well. However, if $L_n\to\infty$, as
$n\to\infty$, this is no longer true as the process $\ZZ_n$ does not
converge weakly. \\

The main point of this paper is to show that, provided
$L_n=(\log n)^ {\gamma} $ for some $0<\gamma<1$,
 the distribution of $\TT_n$  can be estimated by the bootstrap.  The bootstrap counterparts of the above processes are defined as follows.
Let the  bootstrap sample $(\X_1^*,\ldots,\X_n^*)$ be obtained by sampling with replacement from
$\X_1$,$\ldots$, $\X_n$. We write
\begin{eqnarray}
\HH_n^*(\x)=\frac1n \sum_{i=1}^n \1\{ \X_i^*\le \x \},\ \x\in\RR^d,
\end{eqnarray}
 for the empirical cdf based on the bootstrap, with marginals
 \begin{eqnarray}
 \FF_{n,j}^*(x)= \frac1n \sum_{i=1}^n \1 \{ X_{i,j}^*\le x\}, \ x\in\RR,\
  j=1,\ldots,d.
 \end{eqnarray}
We denote its associated empirical copula function by $\CC_n^*$.
The bootstrap empirical copula process is
\begin{eqnarray}
\label{Z_n^*}
 \ZZ_n^*= \sqrt{n}(\CC_n ^*-\CC_n) = \sqrt{n}\left\{ \HH_n^*(\FF_{n,1}^{*-},\ldots, \FF_{n,d}^{*-} ) -
 \HH_n(\FF_{n,1}^{-},\ldots, \FF_{n,d}^{-} )\right\}.
\end{eqnarray}

\medskip

{\sc Assumptions.}
We will assume the following set of assumptions:
\begin{itemize}
\item[(C1)]
For any $j=1,\ldots,d$, for all $\u \in [0,1]^d$ with $0<u_j<1$, the first-order partial derivative $C_j(\u)=\partial C(\u) / \partial u_j$ exists and is of bounded variation (Hildebrandt, 1963, e.g.).
Moreover, it satisfies, for some $r>0$, $\beta \geq 0$ and $K<\infty$,
\[ | C_j(\u) - C_j(\v) | \le K \left( u_j^{-\beta} (1-u_j)^{-\beta} + v_j^{-\beta}(1-v_j)^{-\beta} \right)\sum_{l=1}^d |u_l-v_l|^r,
\]
for all $\u, \v \in [0,1]^d$, $0<u_j,v_j<1$.
As in Segers (2013), we extend the domain of each $C_j$ to the whole $[0,1]^d$ by setting
\[
C_j(\u)  : =
\left\{
\begin{array}{ll}
\limsup_{h\downarrow0} \frac{ C(\u+h{\bf e}_j)}{h}   &   \text { if } \u\in [0,1]^d, u_j =0;\\
\limsup_{h\downarrow0} \frac{ C(\u)- C(\u - h{\bf e}_j)}{h}   &\text{ if }  \u\in [0,1]^d, u_j =1.
\end{array}%
\right.  \]
Here ${\bf e}_j$ is the $j$th coordinate vector in $\RR^d$.
\\

\item[(C2)] The number $L_n$ is of order $(\log n)^\gamma $ for some $0<\gamma<1$.
\\
\end{itemize}

{\sc Remark.} We know that continuity of the  partial derivatives of
$C$ on $(0,1)^d$ is required for weak convergence, see Fermanian et
al. (2004) and Segers (2012). The requirement that the partial
derivatives are of bounded variation is natural since we compute the
supremum of $\ZZ_n$ over increasingly finer families of boxes in
$[0,1]^d$. The process  $ \ZZ_n(\u)$ is asymptotically equivalent to
$\alpha_n(\u)-\sum_{j=1}^d C_j(\u)\alpha_{n,j}(u_j)$ with
$\alpha_n(\u)= \sqrt{n}(\HH_n-H)(\u)$ and $\alpha_{n,j}(u_j)= \sqrt{n}(\FF_{n,j}-F)(u_j)$
 (see Proposition
\ref{prop:Z_n}). \\

{\sc Remark.} The additional requirement (C1) is weaker than imposing a H\"older
condition on the derivatives. Segers (2012) imposes a slightly
stronger condition on the second-order partial derivatives of $C$
(corresponding to $r=1$) to obtain an almost sure representation of
the empirical copula process.

    Indeed, as a counterexample, consider the bivariate Archimedean copula $C$ whose generator is given by $\psi:(0,1]\rightarrow \RR^+$, $\psi(t) := \exp(t^{-\theta}) -e$ for some $\theta>0$.
    This copula, see display (4.2.20) in Nelsen (2006), is
    $$ C(u_1,u_2)=\left[ \ln\left( \exp(u_1^{-\theta})+\exp(u_2^{-\theta}) -e\right)\right]^{-1/\theta},$$
    for any $\u\in[0,1]^2$.
    It can be checked easily that, when $u\rightarrow 0$, the copula density
    $$ C_{12}(u,u) \sim \frac{\theta^2}{4}u^{-\theta - 1}.$$
    Therefore, $C$ cannot fulfill Condition 4.1 in Segers (2012). Nonetheless, by the mean value theorem and simple calculations, we can prove that
    $$ | C_1(\u) - C_1(\v) | \leq K (\min(u_1,v_1))^{-2\theta - 2} |u_1 - v_1| + K (\min(u_1,v_1))^{-\theta - 1}|u_2 -v_2|.$$
    Since the same reasoning can be done with $C_2$, our condition (C1) is fulfilled.\\

 The second assumption (C2)  allows for sub-logarithmic rate in the sample size for the number of boxes considered.  In practice, even this fairly slow rate yields much better tests, see our simulations in Section 4. And we have not observed any significant differences empirically between choosing $\gamma=1$ and $\gamma$ closed to one.\\

Our first result states that the processes $\ZZ_n$ and $\ZZ_n^*$ are close in the bounded Lipschitz distance
that characterizes  weak convergence.
Formally, we show that
\begin{eqnarray}\label{w}
\EE \left[  \sup_{h} \left| \EE [ h(\ZZ_n) ] - \EE^* [ h(\ZZ_n^*) ] \right| \right]
\end{eqnarray} is asymptotically negligible. Here $\EE^*$ is the conditional  expectation with respect to the bootstrap sample and
the supremum in (\ref{w}) is taken over $BL_1= BL_1(\ell^\infty(\F_n))$, the class of
 all uniformly bounded, Lipschitz   functionals $h:\ell^\infty(\F_n)\to\RR$ with Lipschitz constant 1, that is,
\begin{eqnarray}\label{BL1}
\sup_{x\in \ell^\infty(\F_n)} |h(x)|\le 1\end{eqnarray}
 and, for all $x,y\in \ell^\infty(\F_n)$,
\begin{eqnarray}
|h(x)-h(y)|\le   \sup_{f\in\F_n} | x(f)- y(f) |.\label{BL2}
\end{eqnarray}

\begin{thm}\label{Thm1}
Let $\ZZ_n=\{\ZZ_n(f),\ f\in \F_n\}$ and $\ZZ_n^*=\{\ZZ_n^*(f),\ f\in \F_n\}$ with $\F_n$ as defined in (\ref{F:class}) above.
Under conditions (C1) and (C2), we have
\begin{eqnarray}
\lim_{ n\to\infty} \EE\left[ \sup_{h\in BL_1 } \left| \EE [h(\ZZ_n)] - \EE^* [h(\ZZ_n^*)] \right| \right]=0.
\end{eqnarray}
\label{ThFondam}
\end{thm}

\begin{cor}
Under conditions (C1) and (C2) and for any Lipschitz
functional $\phi :\ell^{\infty }(\mathcal{F}_{n})\rightarrow \RR$, we have
\[
\lim_{n\rightarrow \infty }\EE\left[ \sup_{g}|\EE[g(\phi (\ZZ_{n}))]-E^{\ast
}[g(\phi (\ZZ_{n}^{\ast }))]\right] =0.
\]
The supremum is taken over all uniformly bounded Lipschitz functions
$g:\RR\rightarrow \RR$ with $\sup_{x}|g(x)|\leq 1$ and $|g(x)-g(y)|\leq |x-y|$.
\label{cor_Z_n}
\end{cor}

Corollary~\ref{cor_Z_n} follows directly from Theorem~\ref{ThFondam} since, for a fixed Lipschitz function $\phi$, the set  of compositions $g\circ \phi$ above  is a class of uniformly  bounded Lipschitz functions (with the same
  Lipschitz constant).
In particular, since  the mapping $\phi(X)= \sup_{f\in\F_n} |X(f)|$ is Lipschitz, Corollary~\ref{cor_Z_n} implies that
 we can  approximate the distribution of the statistic $\TT_n$ by the conditional (bootstrap) distribution of
\begin{eqnarray}
\label{T_n^*}
\TT_n^*= \sup_{f\in\F_n} | \ZZ_n^*(f)|= \sup_{B_{1},\ldots ,B_{L_n}}\sum_{k=1}^{L_n}|\mathbb{Z}_{n}^*(B_{k})|.
\end{eqnarray}

\begin{cor}
Under conditions (C1) and (C2), we have
 \begin{eqnarray}
\lim_{n\to\infty}\EE\left[ \sup_{g} \left| \EE\left[  g\left(\TT_n \right)\right] -\EE^*\left[g\left(  \TT_n^* \right)\right]\right| \right]=   0.
  \end{eqnarray}
The supremum is taken over all uniformly bounded Lipschitz functions $g:\RR\to\RR$ with $\sup_x |g(x)| \le 1$ and $|g(x)-g(y)|\le |x-y|$ for all $x,y\in\RR$.
\label{corT_n}
\end{cor}

Actually, $\TT_{n}$ and $\widetilde \TT_n$ are   just  two examples of many
potentially useful asymptotic variation type statistics. We mention two other possible statistics:
\begin{itemize}
\item
\textbf{Generalized} $\chi^{2}$ \textbf{statistics}.
Form an equidistant  grid $i/p$, $i=0, \ldots,p =  \left\lfloor L_{n}^{1/d}\right\rfloor +1$
on each axis of $[0,1]^d$, and use the $(p+1)^d$ points of the resulting  equidistant grid on $[0,1]^d$ as the corners of $p^d$ disjoint boxes $B_i$. We define the  statistic $\sum_i | \ZZ_{n}(B_i)|^2$, which, for fixed $L_{n}$,  reduces to a non-normalized $
\chi^{2}$ statistics, in the same spirit as in Dobri\'{c} and Schmid (2005). Here, since the statistic as a function of $\ZZ_n$  is Lipschitz on $\ell^\infty(\F_n)$,   $L_{n}\rightarrow \infty$ is allowed.
 However,  we  suspect that the
full power of Theorem 1 is not needed, since
Radulovi\'c (2013) proved a result similar to Theorem 1 via a more
direct approach, in the non-copula, i.i.d. setting  under  a weaker restriction on the  partition size.\\

\item

\textbf{Generalized Kuiper statistics. }
We start  with the usual Kuiper statistics
\[ K_{1}=\ZZ_n(B_1) =
\sup_{B}      |{\ZZ}_{n}(B)|,\] where supremum is taken over all boxes $B\subseteq[0,1]^d$, and
achieved at $B_{1}$. Then we define recursively, given boxes $B_{1},...,B_{m}$ with $m<L_{n}$,
\[
K_{m+1}=\ZZ_n(B_{m+1})=\sup_{B\cap B_{j}=\emptyset, \, j=1,\ldots,m }|
{\ZZ}_{n}(B)| .\]
The supremum is taken over all boxes $B$ that are disjoint with $B_1,\ldots, B_m$, and  we denote by $B_{m+1}$ for the box at which supremum is achieved. The resulting sum $\sum_{j=1}^{L_n} K_j$ of statistics $K_j$, based on disjoint boxes $B_j$,
is a Lipschitz functional of $\ZZ_n$ and Corollary 2 applies  to this statistics  as well.

\end{itemize}

The performance and the actual implementation of these additional
statistics will not be discussed here, but we will report on them elsewhere.
This paper offers a numerical study only as a proof of principle and for
this purpose we used  the straightforward  statistic $\widetilde{\TT}_{n}$  and  optimization scheme (pure random search) to demonstrate the applicability of Theorem~\ref{ThFondam}.
Nevertheless, even this conservative approach resulted in a
superior performance.\\


\bigskip


\textsc{Remark. }
We may approximate the $\alpha$-upper point of the statistic $\TT_n$  by that of the bootstrap counterpart $\TT_n^*$. Unlike the classical bootstrap situation that assumes a continuous limiting distribution
function,
the bootstrap quantile approximation can be used as follows. Let
$\varepsilon >0$ be arbitrary (independent of $n$) and define the  Lipschitz function
\[ g
_{t,\varepsilon }(x)=\1\{x\leq t\}+\frac{t+\varepsilon -x}{\varepsilon }%
\1\{t<x\leq t+\varepsilon \}.
\]
We have, for  $\delta_n:=\sup_h \left| \EE[h(\TT_{n})]- \EE^*[h(\TT_{n}^{\ast })]\right|$ with the supremum taken over all $h\in BL_1$, uniformly in $t\in\RR$,
\begin{eqnarray*}
\PP\left\{ \TT_{n}\leq t \right\} 
&=&
\EE^{\ast }\left[ g_{t,\varepsilon }(\TT_{n}^{\ast }) \right] + \EE\left[ g _{t,\varepsilon }(\TT_{n}) \right] - \EE^{\ast
}\left[ g _{t,\varepsilon }(\TT_{n}^{\ast }) \right]\\
&
\leq & \PP^{\ast }\left\{ \TT_{n}^{\ast }\leq t+\varepsilon  \right\}+ \delta_n/\eps,
\end{eqnarray*}
since $ g_{t,\eps}$ has Lipschitz constant $1/\eps$. A similar computation shows that
$
\PP^{\ast }\left\{ \TT_{n}^{\ast }\leq t-\varepsilon \right\} - \delta_n/\eps \le
\PP\left\{ \TT_{n}\leq t\right\},
$
so that, uniformly in $t$, and each $\eps>0$
\begin{eqnarray}
\PP^{\ast }\left\{ \TT_{n}^{\ast }\leq t-\varepsilon \right\} - \delta_n/\eps \leq
\PP\left\{ \TT_{n}\leq t\right\}
\leq \PP^{\ast }\left\{ \TT_{n}^{\ast }\leq t+\varepsilon \right\} + \delta_n/\eps
\end{eqnarray}
and in the same way we may prove
\begin{eqnarray}
\PP \left\{ \TT_{n} \leq t-\varepsilon \right\} - \delta_n/\eps \leq
\PP^\ast\left\{ \TT_{n}^\ast \leq t\right\}
\leq \PP \left\{ \TT_{n}\leq t+\varepsilon \right\} + \delta_n/\eps,
\end{eqnarray}
uniformly in $t$, and each $\eps>0$.
For instance, if $t^*$ is the bootstrap $95\%$ critical value of $\TT_n^*$, it is prudent to reject the null for values of $\TT_n$ larger than $t^*+ \eps$.   \\


{\sc Remark.} The test for ${\rm H}_0:\ C=C_0$ based on the critical regions $\{\TT_n > c\}$ is consistent. Indeed, under the null,
since $ \TT_n\le L_n \sup_B |\ZZ_n(B)| $,
we have
$L_n^{-1} \TT_n$ is bounded in probability, while under the alternative hypothesis,
${\rm H}_A\,:C=C_1$ for a fixed $C_1\ne C_0$, we have that
$ \TT_n \ge \sqrt{n} | C_0(B)- C_1(B) | - |\ZZ_n(B)| $, so that
$n^{-1/2} \TT_n \ge \frac12 |C_0(B)- C_1(B) | $, with  probability tending to one,
 for any box $B$ where $C_0$ and $C_1$ differ. Such a box exists under the alternative and the increasing sequence $\F_n$ likely contains at least one such box for relatively small $n$. The improved  power
 of our  test statistic  is confirmed in our simulation
study.
\\

\section{Parametric hypothesis}
\label{Parametric}
\setcounter{equation}{0}

In this section we consider the problem of
testing if  the underlying copula $C$ belongs to a parametric family
$\C:=\{ C_{\theta},\ \theta\in \Theta \}$. That is, the null hypothesis states that
 $C=C_{\theta_0}$ for some $\theta_0\in\Theta$. Here $\Theta \subset \RR^p$, equipped with the Euclidean norm $\|\cdot\|_2$.
 Suppose that we
have a consistent estimator $\wh\theta=\wh\theta(\HH_n)$ of
$\theta_0$.\\

Replacing $C_0$ by $C_{\wh\theta}$ in the definition of the test statistic $\TT_n$, we consider the process
\begin{eqnarray}
\label{Y_n}
\YY_n=
\sqrt{n}(\CC_n -C_{\wh\theta}) = \ZZ_n -
\sqrt{n}(C_{\wh\theta}- C),
\end{eqnarray}
 and its bootstrap version
\begin{eqnarray}
\label{Y_n^*}
\YY_n^* = \ZZ_n^* - \sqrt{n}(C_{\wh\theta^*} - C_{\wh\theta}),
\end{eqnarray}
based on the {\em
non-parametric} bootstrap estimate
$\wh\theta^*=\wh\theta(\HH_n^*)$,  obtained
after  resampling with replacement from the original sample.
Note that
\begin{eqnarray}
\YY_n^*=  \sqrt{n}(\CC_n ^*-C_{\wh\theta^*}) - \sqrt{n}(\CC_n - C_{\wh\theta})\ne \sqrt{n}(\CC_n ^*-C_{\wh\theta^*}).
\end{eqnarray}
The  process $\sqrt{n}(\CC_n ^*-C_{\wh\theta^*})$, while perhaps a natural candidate, does not yield a consistent estimate of the distribution of $\YY_n$. Indeed, the ``distance" between $\YY_n$ and the latter process will be of the order of $\ZZ_n^*$, thus asymptotically tight. On the other hand, the distance between $\YY_n$ and $\YY_n^*$ will be of the same order of magnitude as the distance between $\ZZ_n$ and $\ZZ_n^*$, that tends to zero (see the proof of Theorem~\ref{Thm1}).\\

We stress that our approach does not involve the {\em parametric}
bootstrap, as studied by Genest and R\'emillard (2008), to estimate
the limiting law of copula-based statistics. In other words, we
calculate $\wh\theta^*$ after resampling from the empirical
distribution
$\HH_n$, and not from the law given by the parametric copula $C_{\wh\theta}$.\\

We impose some regularity on
our parameter estimate $\wh\theta$.
\begin{itemize}
\item[(C3)] There exists a   $\psi:\RR^d \mapsto \RR^p$ with  $\int\| \psi\|_2^4 \, {\rm d}H<\infty$ such that
$$ \wh\theta - \theta_0 = \int \psi \, d(\HH_n - H) + \eps_n,\; \text{and} \;\;
 \wh\theta^* - \wh\theta = \int \psi \, d(\HH_n^* - \HH_n) + \eps^*_n,$$
under the null hypothesis, with $\|\eps_n\|_2=o_p(n^{-1/2}/L_n)$ and
$\|\eps^*_n\|_2=o_{p^*}(n^{-1/2}/L_n)$ in probability.
\end{itemize}
Note that the
estimators satisfying (C3) are closely related to the estimators  in the class $\R$
of regular estimators, as defined by Genest and R\'emillard (2008).\\

{\sc Example}
 (Estimators based on the inversion of Kendall's tau).
As an example, we verify condition (C3) for estimators based on the inversion of Kendall's tau in the bivariate case ($d=2$).
Let $\theta=g(\tau)$ for some twice differentiable function $g$ and Kendall's
$\tau:=4\EE[C_\theta(U,V)]-1,$ with the expectation   taken over   $(U,V)\sim C_\theta$.
 Kendall's $ \tau $ is estimated empirically by
$$\wh \tau_n :=
\frac{4}{n(n-1)} \sum_{i=1}^n \sum_{j=i+1}^n \1\left\{
(Y_j-Y_i)(X_j-X_i) >0 \right\} -1.$$ Then $U_n := \wh\tau_n +1$ is a
U-statistic of order 2 for the kernel
\[ h\left( (x_1,y_1); (x_2,y_2) \right) = 2 \cdot \1\{ (y_2-y_1)(x_2-x_1)>0\}.\]
The projection of $U_n- \EE[U_n]$ onto the space of all statistics
of the form $\sum_{i=1}^n g_i(X_i,Y_i)$, for arbitrary measurable
functions $g_i$ with $\EE[g_i^2(X,Y)]<\infty$, is
\[ \wh U_n= \sum_{i=1}^n \EE [ U_n-\EE[U_n] \, |\, X_i,Y_i] =\frac2n \sum_{i=1}^n \{ \psi(X_i,Y_i)-\EE[ \psi(X_i,Y_i)] \}
\]
with
\[ \psi(x,y) =   P(X<x,Y<y)+P(X>x,Y>y).
\]   By H\'ajek's projection principle,
\[ \text{Var} (U_n-\wh U_n) = \text{Var}(U_n) - \text{Var}(\wh U_n).
\]
From the proof of Theorem 12.3 in Van der Vaart (1998), due to Hoeffding (1948),
\[ \text{Var}(U_n) - \text{Var}(\wh U_n) = \frac{4(n-2)}{n(n-1)} \zeta_1 + \frac{2}{n(n-1)} \zeta _2 - \frac{4}{n} \zeta_1 =  \frac{2\zeta_2 -4\zeta_1}{n (n-1)}  \]
with $\zeta_1=\text{Cov}( h(X, Y_1), h(X,Y_2))$ for  $X$ independent of $Y_1$ and $Y_2$, and with the same distribution as $X_1$, and $\zeta_2=\text{Var}( h(X_1, Y_1))$.
Thus the difference is $\text{Var}(U_n) - \text{Var}(\wh U_n)$ is of order $O(1/n^2)$.
Consequently, $U_n-\EE[U_n] =\wh U_n+R_n$ with $R_n=O_p(1/n)$ so that
\[ \wh \tau_n -\tau = U_n- \EE[U_n] = \wh U_n + R_n = \frac{2}{n} \sum_{i=1}^n  \{ \psi(X_i,Y_i)-\EE[ \psi(X_i,Y_i)] \} +
O_p(1/n).
\]
Hence, if $g$ is twice continuously
differentiable in the neighborhood of $\tau$, a limited expansion
ensures that $\hat\theta:= g(\wh \tau_n)$ satisfies the first part of (C3). The second (bootstrap) part of (C3) follows from the same reasoning:
We set $\wh \tau_n^* := U_n^*-1$ with
$$ U_n^* = \frac{4}{n(n-1)} \sum_{i=1}^n \sum_{j=i+1}^n \1\left\{
(Y_j^*-Y_i^*)(X_j^*-X_i^*) >0 \right\} $$ and for
\[ \wh U_n^*= \sum_{i=1}^n \EE^* [ U_n^*-\EE^*[U_n^*] \, |\, X_i^*, Y_i^*]
\]
we can show that
\[ \text{Var}^* (U_n^*-\wh U_n^*) = \text{Var}^*(U_n^*) - \text{Var}^*(\wh U_n^*)
\] is of order $O(1/n)$ almost surely, using the same arguments as above, keeping in mind that the empirical counterparts of $\zeta_1$ and $\zeta_2$ are bounded everywhere.
Moreover, for
\[ \psi_n(x,y) =   \frac1n \sum_{i=1}^n \1\{X_i<x,Y_i<y\} + \frac1n \sum_{i=1}^n \1\{ X_i>x,Y_i>y\},
\] we find
\begin{eqnarray*}
 \wh U_n^* &= &\sum_{i=1}^n \EE^* [ U_n^*-\EE^*[U_n^*] \, |\, X_i^*, Y_i^*]  \\
 &=& \frac{2}{n} \sum_{i=1}^n \psi_n( X_i^*,Y_i^*) - \EE^*[\psi_n( X_i^*,Y_i^*)]\\
 &=&  \frac{2}{n} \sum_{i=1}^n \left\{ \psi( X_i^*,Y_i^*) - \EE^*[\psi( X_i^*,Y_i^*)]\right\} +
  \frac{2}{n} \sum_{i=1}^n \left\{ (\psi_n-\psi)( X_i^*,Y_i^*) - \EE^*[(\psi_n-\psi)( X_i^*,Y_i^*)]\right\} .
\end{eqnarray*}
The second term on the right is of order $O_{p^*}(1/n)$ as its variance equals
\[ \frac4n \text{Var}^* \left ( (\psi_n-\psi) (X_1^*,Y_1^*)\right) \le \frac4 n \sum_{i=1}^n  (\psi_n-\psi)^2( X_i,Y_i) =  O_{p^*}(1/n^2),
\] by the reasoning in Bickel and Freedman (1981, p.1202).
This implies
\[
 \wh \tau_n^* -\wh \tau_n = U_n^*- \EE^*[U_n^*] =  \frac{2}{n}\sum_{i=1}^n  \{ \psi(X_i^*,Y_i^*)-\EE^*[ \psi(X_i^*,Y_i^*)] \}
 + O_{p^*}(1/n).
\]
Again, for a  $g$ that  is twice continuously
differentiable in the neighborhood of $\tau$, a limited expansion
ensures that $\hat\theta^*$ satisfies the second part of (C3).
\\

Moreover, we need more regularity concerning $\theta\mapsto
C_\theta$ itself.
\begin{itemize}
\item[(C4)] For every $(s,t)\in [0,1]^d$, the function $\theta \mapsto C_\theta(\u)$ has continuous partial derivatives
$\dot{C}_\theta(\u)= (\partial /\partial \theta) C_\theta(\u)$
that satisfy a H\"older condition with H\"older exponent $\nu>0$
locally: there exists a constant $K<\infty$ such that
$$ \sup_{\u} \| \dot{C}_{\theta}(\u)- \dot{C}_{\theta_0}(\u)\|_2 \leq K \| \theta
-  \theta_0 \|_2 ^\nu, $$ for every $\theta$ in a neighborhood of
$\theta_0$. Moreover, $\dot{C}_{\theta_0}$ is of bounded
variation.
\end{itemize}

The regularity condition (C4) is satisfied for most of standard
copula families. Simple calculations show that it is the case for
the Gaussian-, Clayton- and the Frank-copula families in particular.
Although copula partial derivatives with respect to their arguments
often exhibit discontinuities or non-existence near their boundaries,
justifying conditions such as (C1) (see Segers, 2012), the derivatives
$\partial C_\theta(x,y)/\partial \theta$ with respect to the copula
parameter  $\theta$ behave a lot more regularly.

\begin{thm}
\label{CoreThParam} Let $\YY_n=\{\YY_n(f),\ f\in \F_n\}$ and
$\YY_n^*=\{\YY_n^*(f),\ f\in \F_n\}$ with $\F_n$ in
(\ref{F:class}) as defined above. Assume that conditions (C1),
(C2), (C3) and (C4) hold.  Then, under the null hypothesis
$\H_0:\, C=C_{\theta},\ \theta\in\Theta$, we have
\begin{eqnarray}
 \lim_{n\to\infty}  \EE\left[ \sup_{h\in BL_1 } \left| \EE[ h(\YY_n)] - \EE^* [h(\YY_n^*)] \right|  \right]  =0.\end{eqnarray}
\end{thm}

This result implies that the distribution of  the test statistic
\begin{eqnarray}\wh{\TT}_n=   \sup_{f\in \F_n} \left| \YY_n(f)
\right| =\sup_{B_{1},\ldots ,B_{L_n}}\sum_{k=1}^{L_n}|\YY_{n}(B_{k})|
\end{eqnarray}
can be ``bootstrapped'' by  the distribution of
\begin{eqnarray}\wh{\TT}_n^*= \sup_{f\in\F_n} | \YY_n^*(f)|=\sup_{B_{1},\ldots ,B_{L_n}}\sum_{k=1}^{L_n}|\YY_n^*(B_{k})|.\end{eqnarray}

\begin{cor}
Assume  that
conditions (C1), (C2), (C3) and (C4) hold. Then, under the null hypothesis ${\rm H}_0: C=C_{\theta}$, $\theta\in\Theta$,
\begin{eqnarray} \lim_{n\to\infty} \EE\left[ \sup_{g} \left| \EE [ g (\wh{\TT}_n) ]- \EE^* [ g(\wh{\TT}_n^*)] \right| \right] =0,\end{eqnarray}
with the supremum taken over all Lipschitz functions $g:\RR\to[-1,1]$ with Lipschitz constant 1. \end{cor}

Often, (C3) can be replaced by
\begin{itemize}
\item[(C3')] There exists a  $\psi:\RR^d \mapsto \RR^p$ with $\int \|\psi\|_2 ^4 \, {\rm d}C<\infty$ such that
$$ \wh\theta - \theta_0 =
 \frac{1}{{n}} \sum_{i=1}^n \left\{
\psi(\FF_{n,1}(X_{i,1}),\ldots,\FF_{n,d}(X_{i,d}) ) - \EE[\psi(F_1(X_{i,1}),\ldots , F_d(X_{i,d}))] \right\}
+ \eps_n,
$$
$$
\wh\theta^* - \wh\theta = \frac{1}{ {n}} \sum_{i=1}^n \left\{
\psi(\FF_{n,1}^*(X_{i,1}^*),\ldots,\FF_{n,d}^*(X_{i,d}^*)) -
\psi(\FF_{n,1}(X_{i,1}),\ldots,\FF_{n,d}(X_{i,d}))\right\} + \eps^*_n,$$ under the null
hypothesis, with $\|\eps_n\|_2=o_p(n^{-1/2}/L_n)$ and
$\|\eps^*_n\|_2=o_{p^*}(n^{-1/2}/L_n)$ in probability.
\end{itemize}

This is a consequence of
the following result.
\begin{prop}
Assume
(C1) holds. Any
 estimator  $\wh\theta$   satisfying (C3'), satisfies
(C3). \label{ThC3C3bis}
\end{prop}

Copula parameters are typically estimated  through
pseudo-observations or ranks,  without any assumption on the
marginal distributions.  For this reason  the
copula estimators that satisfy (C3') are relevant. They are very
closely related  to the estimators in the class $\R_1$ of
Genest and R\'emillard (2008). In particular, the maximum
pseudo-likelihood estimator, that maximizes
the pseudo log-likelihood function
$ \int \log c_\theta \, d\CC_n$ over $\theta\in\Theta$,
see, for instance,
Genest et al. (1995) or Shih and
Louis (1995), satisfies (C3') under suitable regularity conditions on the copula density $c_\theta$.\\

Since the bootstrapped copula process $\YY_n^*$ is new, it is noteworthy to stress that it provides a valuable alternative to the usual parametric bootstrap.
Now, assume $L_n=L$ is a constant, to retrieve the standard framework.
\begin{cor}
\label{CorNewBoots} Assume  that conditions (C1), (C3) and (C4)
hold. Then, the process $\{\YY_n(\u), \, \u\in [0,1]^d \}$ tends
weakly towards a Gaussian process  in $\ell^{\infty}([0,1]^d)$.
Moreover, the bootstrapped process $\{\YY_n^*(\u), \, \u \in [0,1]^d \}$ converges weakly to the same Gaussian  process  in probability
in $\ell^{\infty}([0,1]^d)$.
\end{cor}

\section{Applications and Numerical Studies}

 \setcounter{equation}{0}

We present  a limited numerical study,
 serving as a
proof of principle rather than the final word on this subject. The evaluation of GOF tests in copula
settings is a complex problem and only partial answers can be found in literature: see the surveys of  Berg (2009), Genest et al. (2009) and, more recently, Fermanian (2012). Here, we restrict ourselves to the bivariate case. A full-scale numerical analysis is beyond the scope of this paper.

\mds

We have implemented $\widetilde{\TT}_n$, a computationally simpler version of $\TT_n$, see Appendix C for the algorithm.
In the case of a composite null hypothesis, we have implemented a simplified version of $\hat{\TT}_n$ in the same way, by restricting the boxes $B$ to be of the form $B=\prod_{i=1}^d (a_i, b_i]$ with $a_i, b_i \in  \{ n^{-1/d}, 2n^{-1/d}, \ldots\} \subset [0,1]$.
Since the distance  between $\TT_n$ and $\widetilde{\TT}_n$ tends to zero  in probability (as a result of Lemma 9 and Proposition 10 in    Section 5), the weak convergence results are valid with $\widetilde{\TT}_n$ instead of $\TT_n$ or $\hat{\TT}_n$. 
Moreover, the reasoning to approximate p-values by bootstrap still applies.\\

\subsection{Heuristics}
For two copula  densities $c_{0}$ and $c_{1}$, we define the
\textit{difference} sets $A^{+}$ and $A^{-}$ as
\begin{equation*}
A^{+}=\{(s,t):c_{0}(s,t)>c_{1}(s,t)\},\;\text{ and }A^{-}=
\{(s,t):c_{0}(s,t)<c_{1}(s,t)\}.
\end{equation*}
The proposed test statistics
are designed to sample $L_n$ boxes in order to maximize
the difference between the ``true'' and postulated copulas. In
situations where the geometry of the difference sets $A^{+}$ and
$A^{-}$ is complex,
statistics such as $\widetilde\TT_{n}$ can ``pick out"  disjoint subregions of $A^+$ and $A^-$,
and one could expect superior performance
consequently.
However,
sometimes just a single  well placed box can pick essentially all the mass of
sets $A^{+}$ or $A^{-}$, while the remaining $L_n-1$ boxes are just
collecting  noise and consequently  diminish the power of the
statistic $\TT_{n}$.\\

\begin{figure}[!!h]
\centering
\includegraphics*[width=12truecm,height=8truecm]{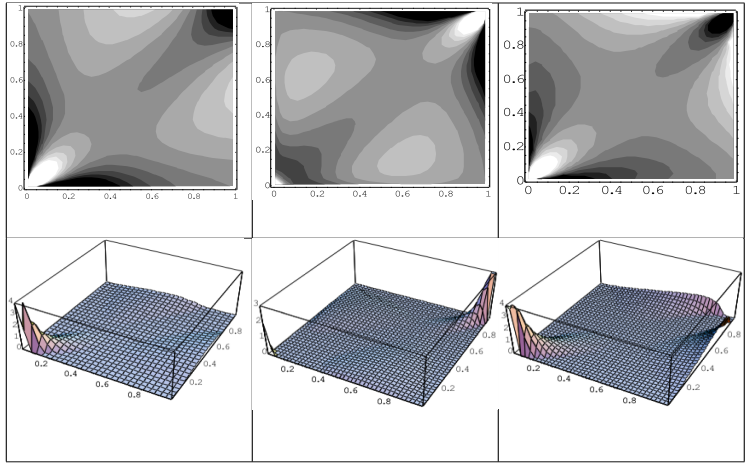}
\caption{\small {\bf Common comparisons}.
Copula density differences, through contour plots and 3D plots of synthetic data: Clayton - Frank (left), Gumbel - Frank (center), Clayton - Gumbel (right). Their Kendall's tau is $0.4$.
} \label{figure1}
\end{figure}

Most common scenarios encountered in the literature compare Frank,
Clayton, Gumbel, and Gauss copulas with each other, after
controlling for some dependence indicator (typically Kendall's tau):
see, for instance, Berg (2009), Genest and R\'emillard (2008) and
Genest et al. (2009). However, all these pairings  produce trivial
difference sets $A^{+}$ and $A^{-}$, as revealed in the contour
plots and 3D plots of  $c_0-c_1$ of Figure 1. We see that nearly all
the mass difference between copula densities $c_{0}$ and $c_{1}$ is
concentrated in a single spot, located in either the lower left or
upper right corner. Here  Kendall's $\tau=0.4$, but we observed
similar plots for  different values of $\tau$. Therefore, these
common simulation scenarios are tailored towards  many standard GOF
tests such as KS and CvM tests.
We are not aware of any argument that justifies such
specific types of pairing, except for  analytical tractability.
Figures \ref{figure2} and \ref{figure3}, however,   paint a
very different scenario  with more elaborate difference sets
$A^{+}$ and $A^{-}$ that appear in real life situations.
How often and to
what extent this complex situation is encountered in reality is largely an open empirical issue.\\

In this study, the  copula densities $c_1$ were estimated by kernel density estimators 
based on the following data:
\begin{itemize}
\item The bivariate \emph{ARCH}-like process $(X_{1},Y_{1}),\ldots,(X_n,Y_n)$, with $n=10^{6}$, was generated as follows:
First, we created independent $Z_i\sim N(0,1)$ and  $W_{i}=Z_{i}  ( 1+0.6W_{i-1}^{2})^{1/2} ,$ with $W_0=0$.
Second, we set $(X_{i},Y_{i}):=(W_{100i},W_{100i+1})$, creating nearly independent couples (of strongly dependent observations). Such models are commonly used in empirical finance, for instance.
\item
The \emph{Mixture Copula}
data $(X_{1},Y_{1}),\ldots,(X_n,Y_n)$, with $n=10^{6}$, are generated from the
mixture $c_{1}(s,t)=\frac12 c_{F}(s,t)+\frac12 c_{F}(1-s,t)$ for the  Frank
copula  $c_{F}$ with Kendall's $ \tau=0.4$. Therefore, this copula has asymmetrical features, contrary to most copulas that are tested in the literature. Obviously, other asymmetrical copulas could be built, following Liebscher (2008) for instance.
\item
The \emph{Euro-Dollar} data $(X_{1},Y_{1}),\ldots,(X_n,Y_n)$, with $n=1800,$ are quoted currency exchange values. $X$ is the
daily percentage change of the  Euro against the  US dollar, while $Y$ corresponds to the
daily change of the  Canadian dollar against the  US dollar.
\item
The \emph{Silver-Gold} data $(X_{1},Y_{1}),\ldots,(X_n,Y_n)$, with $n=5000$, presents
the log ratio of the average daily price of silver and gold futures respectively.
For instance, $X_{i}=\log (S_{i+1}/S_{i})$ based on the   average price $S_{i}$ of
silver in US dollars on day $i$.
\end{itemize}
We compared \emph{Mixture copula} and \emph{ARCH} with the independence copula, for which $c_{0}(s,t)=1$ (see Figure 2).
In the case of real data (\emph{Euro-Dollar} and \emph{Silver-Gold}), we choose the  Frank copula density with parameters $\tau =2.6$ and $\tau
=3.4$, respectively, for $c_0$ (see Figure 3). The latter parameters were chosen after minimizing the (estimated) $L_{1}$-distance between $c_{0}$ and $c_{1}$.
The difference sets are easily depicted by dark and bright sections of the contour plots, and the
3D plots clearly indicate that the mass difference between copula densities $c_0$ and $c_1$ is not concentrated in a single spot. 

\begin{figure}[!!h]
\centering
\includegraphics*[width=12truecm,height=8truecm]{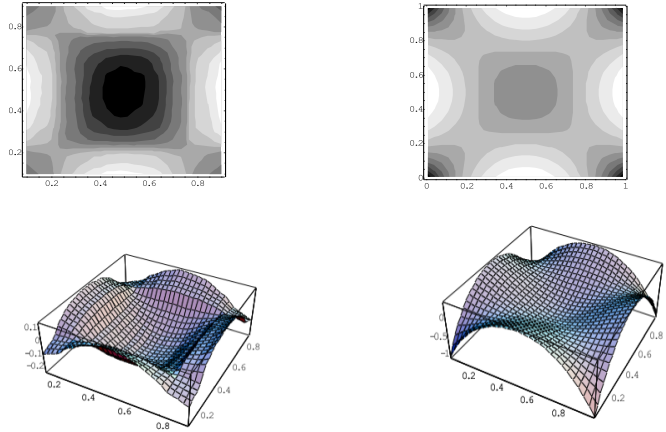}
\caption{\small {\bf Complex relation (synthetic data)}. Copula density differences, through contour plots and 3D plots: ARCH (left) and mixture copula (right), compared to the independence copula.
} \label{figure2}

\end{figure}

\begin{figure}[!!h]
\centering
\includegraphics*[width=12truecm,height=8truecm]{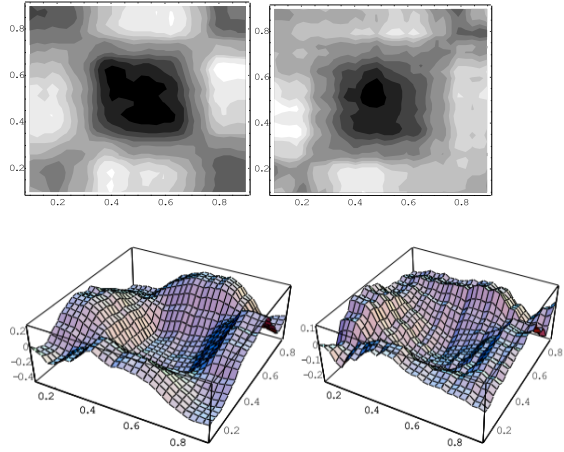}
\caption{\small {\bf Complex relation (actual data)}.
Copula density differences, through contour plots and 3D plots: Euro - Dollar (left) and Silver - Gold (right), compared to Frank copulas (with Kendall's tau equal to $2.6$ and $3.4$ respectively).
} \label{figure3}
\end{figure}

\subsection{GOF tests in practice}

We generated the data sets \emph{ARCH} and \emph{Mixture Copula} as described above.
For each data set, we run two sets of simulations:
\begin{itemize}
\item (ARCH-S and Mixture-S) Test the
simple null hypothesis $C_{0}(s,t)=st$ using the methodology of Section 2.
\item
 (ARCH-C and Mixture-C) Test the composite null
hypothesis that $C_0$ is a Frank copula using the  procedure
 described in Section  3.
 \end{itemize}
In both cases, the null hypothesis is wrong and should be rejected.\\

In our simulations, the number of boxes is $L_n=\left\lfloor \ln^{0.95}(n)\right\rfloor -2.$
 We approximated the  p-values of  all the statistics we consider   via
the  bootstrap procedures introduced in  sections 2 and 3. For each
approximation, we used 1,000 bootstrap samples. For the second set
of simulations (ARCH-C and Mixture-C), we
  computed the parameters $\widehat{\theta }$ and $\widehat{\theta }^{\ast }$ by the usual
pseudo-maximum likelihood procedure. Each procedure is repeated 100
times. We report  the percentage of times that the computed
$p$-value is below  $\alpha =0.05$.\\

Our limited numerical study confirms the above
assessment.
Table 1 shows that the ATV test outperforms largely the  KS and CvM tests in the case of complex pairing, while Table 2 confirms  that the  ATV
test is inferior in  case of the  commonly used  pairings of Figure 1.\\
In Table 2, for each pair of
copulas, say Clayton - Frank, we generated  $n$ observations  from the first copula (Clayton), and we tested the null hypothesis that the second
copula (Frank) is the true underlying copula. 
In this simple scenario,
the sophistication of $\TT_n$ is a disadvantage compared to  simpler usual test statistics. The former test looks for
discrepancies everywhere in the
unit hypercube (at the price of noise), while  the simpler KS and CvM tests  pick up easily the right boxes (by chance, in our opinion).\\
Table 3 shows that  the significance level of the ATV test is below $0.05$. The
data were simulated from the
null hypothesis.
In all tables,
Kendall's  $\tau=0.4$.\\

{

\begin{table}
\begin{equation*}
\begin{array}{cccccc}
\text{type} & n & \text{ARCH-S } & \text{ARCH-C} & \text{Mixture-S } & \text{
Mixture-C} \\
\text{ATV} & 400 & \mathbf{75\%} & \mathbf{80\%} & \mathbf{41\%} & \mathbf{25\%} \\
\text{KS} & 400 & 6\% & 4\% & 8\% & 12\% \\
\text{CvM} & 400 & 25\% & 50\% & 6\% & 15\% \\
\text{ATV} & 800 & \mathbf{100\%} & \mathbf{99\%} & \mathbf{94\%} & \mathbf{98\%}
\\
\text{KS} & 800 & 32\% & 50\% & 20\% & 25\% \\
\text{CvM} & 800 & 50\% & 92\% & 31\% & 84\%
\end{array}
\end{equation*}
\caption{{\bf Complex pairing}, related to Figure 2: relative frequencies of rejected null hypotheses under $\alpha=0.05$.}
\end{table}

\begin{table}\begin{equation*}
\begin{array}{ccccc}
\text{type} & n & \text{Clayton - Frank} & \text{Gumbel - Frank} & \text{
Clayton - Gumbel} \\
\text{ATV} & 400 & 42\% & 26\% & 88\% \\
\text{KS} & 400 & 58\% & 25\% & 90\% \\
\text{CvM} & 400 & \mathbf{84\%} & \mathbf{47\%} & \mathbf{95\%} \\
\text{ATV} & 800 & 92\% & 58\% & 94\% \\
\text{KS} & 800 & 98\% & 53\% & 98\% \\
\text{CvM} & 800 & \mathbf{100\%} & \mathbf{73\%} & \mathbf{100\%}
\end{array}
\end{equation*}
 \caption{{\bf Trivial pairing}, related to Figure 1:
relative frequencies of rejected null
  hypotheses under $\alpha =0.05.$ }
\end{table}

\begin{table}
\begin{equation*}
\begin{array}{ccccc}
\text{type} & n & \text{Clayton - Clayton} & \text{Gumbel - Gumbel} & \text{
Frank -Frank} \\
\text{ATV} & 400 & 3\% & 2\% & 2\% \\
\text{KS} & 400 & 4\% & 5\% & 4\% \\
\text{CvM} & 400 & 4\mathbf{\%} & 5\mathbf{\%} & 4\mathbf{\%} \\
\text{ATV} & 800 & 2\% & 4\% & 3\% \\
\text{KS} & 800 & 3\% & 3\% & 5\% \\
\text{CvM} & 800 & 5\mathbf{\%} & 3\mathbf{\%} & 6\mathbf{\%}
\end{array}
\end{equation*}
\caption{ {\bf Errors of the first kind}: relative frequencies of rejected null hypotheses under $\alpha=0.05$.}
\end{table}
}

\section{Proofs}
\label{myproofs}
\setcounter{equation}{0}

Throughout the proofs, we assume
without loss of generality that $F_j=I$ for every $j=1,\ldots,d$ (uniform marginal distributions). This implies that $H=C$.
This is justified by the following lemma.

 \begin{lemma}\label{uniform_trick}
Let $F_j$, $j=1,\ldots,d$ be continuous distribution functions. Denote by $\tilde H$ the cdf of $(F_1(X_1),\ldots, F_d(X_d))$ and by $\tilde C$ its associated copula.
The empirical copula associated to the sample $(F_1(X_{i1}),\ldots,F_d(X_{id}))$, ${i=1,\ldots,n}$, is denoted by $\tilde{\mathbb C}_n$.
We have $$ C(\u)=\tilde C(\u)= \tilde H(\u)\text{ for all }
\u \in[0,1]^d.$$ Moreover,
\[ {\mathbb C}_n\left(\frac{i_1}{n},\ldots,\frac{i_d}{n}\right)=
\tilde{\mathbb C}_n\left(\frac{i_1}{n},\ldots,\frac{i_d}{n}\right) \text{ for }
i_1,\ldots,i_d \in \{0,1,\ldots, n\}.\]
\end{lemma}
\begin{proof} This is a straightforward extension of Lemma 1 in Fermanian et al. (2004). \end{proof}

Since the letter $C$ is reserved for the copula function, we use the letters $K, K_0, K_1$, etc. in the sequel to denote  generic constants, and we
write  $\|\s\| _\infty = \max_{1\le j \le d} |s_j|$  of $\s=(s_1,\ldots,s_d)\in[0,1]^d$.
\\

\subsection{\sc Proof of preliminary results}

In general, note that, for each $f\in\F_n$ defined in (\ref{F:class}), we can write
\begin{eqnarray*}
\ZZ_n(f)=
\sum_{k=1}^{L_n} c_k \ZZ_n(B_k) =
\sum_{l=1}^{2^d L_n}
\sigma_l \ZZ_n(\s_l),
\end{eqnarray*}
and $$\ZZ_n^*(f)= \sum_{l=1}^{2^d L_n}
\sigma_l \ZZ_n^*(\s_l) ,$$
 for some  $\sigma_l \in\{-1,+1\}$  and  $\s_l \in [0,1]^d$, using formula (\ref{Z_n:box}).
Let $\alpha_n(\u) := \sqrt{n}(\HH_n-H)(\u)= \sqrt{n}(\HH_n(\u)-\u)$ be the ordinary uniform empirical
process in $[0,1]^d$, and let its oscillation modulus be  defined as
\begin{equation}
\label{M_n}
\MM_n(\delta):=  \sup\left\{ \left | \alpha_n(\s) -
\alpha_n(\s') \right|:\ \|\s-\s'\|_\infty\le \delta; \,\s,\s' \in [0,1]^d\right\},
\end{equation}
for any $\delta>0$.\\

\begin{lemma}\label{prop:oscillation}
Let $(\delta_n)_{n\geq 0}$ be a sequence of positive real numbers such that $n\delta_n/\log n\rightarrow \infty$.
Then, we have
\[ \MM_n(\delta_n)= O(\delta_n^{1/2} (\log n)^{1/2} )  \qquad \text{almost surely}.\]
\end{lemma}

\begin{proof}
We apply
Proposition \ref{segers}    with $\lambda_n=K_0 \delta^{1/2}_n (\log n)^{1/2}$ for some constant $K_0>0$.
Since $n^{-1/2}\lambda_n / \delta_n = K_0\left(\log n/(n\delta_n)\right)^{1/2}$ tends to zero, this inequality can be rewritten
$$ \PP\left\{ \MM_n(\delta_n) > \lambda_n \right\} \leq \frac{K_1}{\delta_n} \exp\left( -\frac{K_2\psi(1)\lambda^2_n}{\delta_n}   \right) = K_1 n \exp\left( - K_2 K_0^2 \psi(1) \log n \right),$$
for some constants $K_1$, $K_2$ and $n$ sufficiently large. When $K_0$ is sufficiently large, we check that
$$ \PP\left\{  \MM_n(\delta_n) > \lambda_n \right\}\leq \frac{K_3}{n^2},$$ for some constant $K_3$. Invoke the
 Borel-Cantelli Lemma  to  conclude the proof.
\end{proof}

In addition, let $\alpha_{n,j}(u)= \sqrt{n} (\FF_{n,j}-F_j)(u) = \sqrt{n} (\FF_{n,j}(u)-u)$ be the ordinary uniform (marginal) empirical process in $[0,1]$, and
we define
\begin{equation}
\widetilde \ZZ_n(\s)=\alpha_n(\s)-\sum_{j=1}^d C_j(\s)\alpha_{n,j}(s_j). \label{ZZtilde}
\end{equation}

\begin{prop}\label{prop:Z_n}
Under conditions (C1) and (C2), we have
$$\lim_{n\to\infty}  \sup_{h\in BL_1}  \left|  \EE [ h(\mathbb{Z}_n)]  - \EE[ h( \widetilde \ZZ_n)]  \right| =0.$$
\end{prop}
\begin{proof}
First, we observe that
\begin{eqnarray*}
 \sup_{h\in BL_1} \left|  \EE [ h(\ZZ_n) - h(\widetilde\ZZ_n) ] \right|
&\le& \delta+ 2\PP\left\{ \sup_{f\in\F_n}  | \ZZ_n(f)-\widetilde \ZZ_n(f) | > \delta \right\}.
\end{eqnarray*}
The latter inequality holds for any $\delta>0$, and uses the fact
that $|h|$ is bounded by 1 and has Lipschitz constant 1. It
remains to show that
$$ \sup_{f\in\F_n}  | \ZZ_n(f)-\widetilde \ZZ_n(f) | \to 0,$$
in probability, as $n\to\infty$.
 The remainder of the proof generalizes  Proposition 4.2 of Segers (2012).
Now,
we note that
\begin{eqnarray*}
 \sup_{f\in\F_n} | \ZZ_n (f)- \widetilde \ZZ_n (f) | \le 2^d L_n \sup_{ \s \in [0,1]^d } | \ZZ_n(\s) - \widetilde \ZZ_n(\s) |\le
2^d L_n
(I + II)
\end{eqnarray*}
with
\begin{eqnarray*}
I &=&  \sup_{\s\in [0,1]^d}  \left| \alpha_n(\FF_{n,1}^{-}s_1,\ldots,\FF_{n,d}^{-}s_d) -\alpha_n(\s) \right|,
\\
II&=&  \sup_{\s \in [0,1]^d} \left|  \sqrt{n} \left[
C(\FF_{n,1}^{-}s_1,\ldots,\FF_{n,d}^{-}s_d) -C(\s) \right] + \sum_{j=1}^d C_j(\s)\alpha_{n,j}(s_j) \right|.
\end{eqnarray*}
The first term, $I$, can be bounded as follows. Set $\beta_{n,j}(s)= \sqrt{n} (\FF_{n,j}^{-} s- s)$, $j=1,\ldots,d$. By the Chung-Smirnov LIL, we have
$$\max_{1\le j\le d} \sup_{0\le s\le 1} | \beta_{n,j}(s) |  = O((\log\log n)^{1/2}) \qquad \text{ almost surely}.$$
Using Lemma~\ref{prop:oscillation} with $\delta= n^{-1/2} (\log\log n)^{1/2}$, we get
\[ \sup_{\|\s-\s'\|_\infty < \delta} \left | \alpha_n(\s) - \alpha_n(\s') \right| = O(n^{-1/4} (\log n)^{1/2} (\log \log
n)^{1/4}),
\]
almost surely.
This implies that $I= O(n^{-1/4} (\log n)^{1/2} (\log \log n)^{1/4})$, almost surely.

\medskip

For the second term, we get by the mean value theorem that
\begin{eqnarray*}
II &=&  \sup_{\s\in [0,1]^d } \left|  \sqrt{n}[ C(\FF_{n,1}^{-}s_1,\ldots,\FF_{n,d}^{-}s_d )-C(\s) ]
  + \sum_{j=1}^d C_j(\s)\alpha_{n,j}(s_j) \right|\\
 &\leq &  \sup_{\s\in [0,1]^d } \left|   \sum_{j=1}^d C_j(\s_n) \beta_{nj}(s_j)    + \sum_{j=1}^d C_j(\s)\alpha_{n,j}(s_j) \right|,
\end{eqnarray*}
where $\s_n$ is a vector in $[0,1]^d$ s.t. $\| \s_n- \s\| _\infty\le n^{-1/2}\max_{1\le j\le d}  |\beta_{n,j}(s_j)|$.
Since $|C_j|\le 1$ for every $j=1,\ldots,d$ (because copulas are Lipschitz with Lipschitz constant 1), we deduce
\begin{eqnarray*}
II &\leq &  \sup_{\s\in [0,1]^d } \sum_{j=1}^d \left| \beta_{nj}(s_j)+ \alpha_{n,j}(s_j)\right|
 +  \sup_{\s\in [0,1]^d }\sum_{j=1}^d \left| [C_j(\s_n) - C_j(\s)]\alpha_{n,j}(s_j) \right| \\
& \leq & IIa + IIb.
\end{eqnarray*}
The Bahadur-Kiefer theorem  (Shorack and  Wellner, 2009, p. 585) states that
 \[\max_{1\le j\le d} \sup_{0\le s\le 1}  | \beta_{n,j}(s) + \alpha_{n,j}(s_j)   | = O(n^{-1/4} (\log n)^{1/2} (\log \log n)^{1/4}) \quad \text{almost surely}.
 \]
Then, $IIa = O(n^{-1/4} (\log n)^{1/2} (\log \log n)^{1/4}) $ almost surely.

\medskip

Concerning $IIb$, we consider a positive sequence $(\eps_n)$, $\eps_n\rightarrow 0$, that will be specified later independently of any $\s=(s_1,\ldots,s_d)\in [0,1]^d$.
For any index $j=1,\ldots,d$ and any $\s\in [0,1]^d$, we will distinguish the two cases: $s_j\in [\eps_n,1-\eps_n]$ and the opposite.

\medskip

If $s_j\in [\eps_n,1-\eps_n]$ then
\[ s_{nj}= s_j\left( 1+\frac{s_{nj}-s_j}{s_j} \right) \ge s_j \left( 1- \frac{|s_{nj}-s_j|}{\eps_n} \right)\ge
\frac{s_j}{2},
\]
and
\[ 1-s_{nj} \ge (1-s_j) \left( 1- \frac{|s_{nj}-s_j|}{\eps_n} \right)\ge \frac{1-s_j}{2},
\]  almost surely and for $n$ sufficiently large, for all $\eps_n\to0$ and  $n\eps_n^2 / \log n \to \infty$.
Corollary  2 in Mason (1981) implies that
$$\max_{1\le j\le d} \sup_{0\le s_j \le 1}  |s_j^{-1/2}(1-s_j)^{-1/2}\alpha_{n,j}(s_j) | \leq K(\log n)^{1/2}\log\log n, $$
almost surely, for some constant $K>0$.




In this case, using condition (C1), we deduce,
\begin{eqnarray*}
\left| C_j(\s_n) - C_j(\s)\right|   | \alpha_{n,j}(s_j)| &\leq& K_0 \| \s_n -\s \|^r \left\{s_j^{-\beta} (1-s_j)^{-\beta} + s_{nj}^{-\beta} (1-s_{nj})^{-\beta}
\right\} | \alpha_{n,j}(s_j)| \\
&\leq &  K_1 \| \s_n -\s \|^r s_j^{1/2-\beta} (1-s_j)^{1/2-\beta}(\log n)^{1/2}\log\log n \\
&\leq &  K_2 n^{-r/2}(\log\log n)^{r/2} \max(\eps_n^{1/2-\beta},1) (\log n)^{1/2}\log\log n,
\end{eqnarray*}
almost surely, for some constants $K_0,K_1,K_2>0$ and every $j$.

\medskip

If $s_j\not\in [\eps_n,1-\eps_n]$ then
\begin{eqnarray*}
\left| C_j(\s_n) - C_j(\s)\right| | \alpha_{n,j}(s_j)| &\leq &
2 |\alpha_{n,j}(s_j) | \\
&\le &  2\eps_n^{1/2}  s_j^{-1/2} (1-s_j)^{-1/2} |\alpha_{n,j}(s_j) | \\
&\leq & K \eps_n^{1/2} (\log n)^{1/2}\log \log n \ \text{ almost surely,}
\end{eqnarray*} see Corollary       2 in Mason (1981).

\medskip

Combining all these bounds
entails then
$$ IIb \leq K_3
\left[ n^{-r/2}(\log\log n)^{r/2} \max(\eps_n^{1/2-\beta},1) +  \eps_n^{1/2}  \right] (\log n)^{1/2}\log \log n  ,$$
with $K_3>0$.
We now specify the choice of $\eps_n=n^{-p}$, with $p$ depending on $\beta$ and $r$ only.
If   $2\beta > 2r+1$,  we take $0< p < r/(2\beta - 1)$. If $\beta < 1/2$, set $p=1/4$.
Otherwise, take $p=\min(1/4,r/(4\beta - 2))$, for instance.
In each case, these choices ensure that  $IIb = O( n^{-q} )$ almost surely, for some $q>0$.

\medskip

Since $L_n=O(\log n)$ by assumption (C2), we obtain $L_n ( I+II) \to0$ almost surely, as $n\to\infty$, and  the proof is
complete.
\end{proof}

Next, we turn our attention  to the bootstrap counterparts. We define $\alpha_n^*(\s)= \sqrt{n}(\HH_n^*-\HH_n)(\s)$ as  the ordinary bootstrap  empirical process in $[0,1]^d$.
We prove the following exponential inequality for the oscillation modulus
\[ \MM_n^*(\delta) = \sup_{\|\s-\s'\|_\infty<\delta}  |\alpha_n^*(\s)-\alpha_n^*(\s')|.\]

\begin{lemma} \label{lemma:alpha_n^*}
For all bounded sequences $\delta_n$ such that $n\delta_n/ \log(n) \to\infty$ as $n\to\infty$,
\begin{eqnarray}\label{M_n^*}
 \MM_n^*(\delta_n)= O(\delta_n^{1/2} (\log n)^{1/2} )\ \quad \text{almost surely.}\end{eqnarray}
\end{lemma}
Note that the sequence $(\delta_n)$ may be constant.
\begin{proof}
Since $\alpha_n^*$ is a step function, we find that
$$\sup_{\ \|\s-\s'\|_\infty< \delta_n } |\alpha_n^*(\s)-\alpha_n^*(\s')|=\max  |\alpha_n^*(X_{i_1,1},\ldots,X_{i_d,d})-\alpha_n^*(X_{i'_1,1},\ldots,X_{i'_d,d})| ,$$
with the maximum taken over all  $|X_{i_j,j}-X_{i'_j,j}|< \delta_n$, $j=1,\ldots,d$, $ i_1,i'_1,\ldots,i_d,i'_d \in \{1,\ldots,n\}$.
For any  $\i:=(i_1,\ldots,i_d)$ and $\i'=(i_1',\ldots,i_d')$ in $\{1,\ldots,n\}^d$,  we rewrite
$$|\alpha_n^*(X_{i_1,1},\ldots,X_{i_d,d})-\alpha_n^*(X_{i'_1,1},\ldots,X_{i'_d,d})| = n^{-1/2}\sum_{k=1}^n \{ V_{k,\i,\i'} - \EE^*  [V_{k,\i,\i'}] \},$$
as a sum of bounded independent random variables
with
$$ V_{k,\i,\i'}:=  \1\{X_{k,j}^*\leq X_{i_j,j}, \, j=1,\ldots,d\}- \1\{X_{k,j}^*\leq X_{i'_j,j},\,   j=1,\ldots,d\},$$
conditionally on the sample $(\X_1,\ldots, \X_n)$.
Moreover, a simple calculation and Lemma~\ref{prop:oscillation} yield
\begin{eqnarray*}
\text{Var}^* (V_{k,\i,\i'}) &\le& \sum_{j=1}^d \PP^*\left\{ 
\min(X_{i_j,j}, X_{i_j',j} ) \le  X_{k,j}^*\le \max
(X_{i_j,j}, X_{i_j',j} )\right\}  \\
&\le&\sum_{j=1}^d  \sup_{s_j} [\FF_{n,j}(s_j+\delta_n) - \FF_{n,j}(s_j)] \\
&\le& d\delta_n + dn^{-1/2} \MM_n(\delta_n)\\
&\le&  d \max(\delta_n,\MM_n(\delta_n)/\sqrt{n})
\\
&\le& K \max(\delta_n,\sqrt{\delta_n \log n}/\sqrt{n})= K\delta_n,
\end{eqnarray*}
for $n$ large enough, for almost all realizations   and for some constant $K>0$.
Hence, by the union bound and Bernstein's exponential inequality for bounded random variables,
we have, for some constant $K_0$,
\begin{eqnarray*}
&& \PP^*\left\{\max_{\substack{\i,\i' \in \{1,\ldots,n\}^d \\ |X_{i_j,j}-X_{i'_j,j}|< \delta_n,\; \forall j} }
|\alpha_n^*(X_{i_1,1},\ldots,X_{i_d,d})-\alpha_n^*(X_{i'_1,1},\ldots,X_{i'_d,d})|> x \right\}\\
&&\le  2n^{2d} \exp\left( - K_0 ( \sqrt{n}x \wedge x^2 \delta_n^{-1} ) \right),
\end{eqnarray*}
for all samples $(\X_1,\ldots,\X_n)$.
By integrating the previous inequality over $\PP$, we get the same inequality, but replacing $\PP^*$ by $\PP$.
Set $x=K_1\delta_n^{1/2}(\log n)^{1/2}$ and take a constant $K_1$ sufficiently large to obtain
$$ \sum_{n=1}^{+ \infty}\PP\left\{ \MM_n^*(\delta_n) > K_1\delta_n^{1/2}(\log n)^{1/2} \right\} < +\infty.$$
Apply the Borel-Cantelli lemma to conclude the proof.
\end{proof}

\medskip

Analogous to the approximation of the process
 $\ZZ_n$ by $\widetilde \ZZ_n$ before,   we introduce a simpler process $\widetilde \ZZ_n^*$ to approximate $\ZZ_n^*$.
Set
\begin{equation}
\widetilde \ZZ_n^*(\s)=  \sqrt{n}(\HH_n^*-\HH_n)(\s) - \sum_{j=1}^d C_j(\s)
\sqrt{n}(\FF_{n,j}^*-\FF_{n,j})(s_j).
\label{ZZ_ntildeast}
\end{equation}

\begin{prop}\label{prop:Z_n^*}
Under conditions (C1) and (C2), we have
$$\lim_{n\to\infty} \EE \left[ \sup_{h\in BL_1} \left|  \EE^* [ h (\mathbb{Z}_n^*) - h( \widetilde \ZZ_n^*) ] \right| \right] =0.
$$
\end{prop}
\begin{proof}
First, we notice that, for any $\eta>0$,
\begin{eqnarray*}
\EE\left[ \sup_{h\in BL_1} \left|  \EE^* [ h (\mathbb{Z}_n^*) - h( \widetilde \ZZ_n^*) ] \right| \right]  &\le& \eta+ 2\EE\left[ \PP^*\left\{ \sup_{f\in \F_n} \left| \ZZ_n^*(f) - \widetilde \ZZ_n^*(f) \right| \ge \eta\right\}\right]\\
  &\le& \eta+ 2\EE\left[ \PP^*\left\{ \sup_{{\bf s}} 2^d L_n  \left| \ZZ_n^*({\bf s}) - \widetilde \ZZ_n^*({\bf s}) \right| \ge \eta\right\}\right].
\end{eqnarray*}
Some straightforward adding and subtracting yields
$\ZZ_n^*(\s) = \bar \ZZ_n^*(\s) + R_n^*(\s)
$ 
with
\[ \bar \ZZ_n^*(\s)= \sqrt{n}\{ \HH^*_n(\s)-\HH_n(\s)\} - \sqrt{n}\{ C(\FF_{n,1}^*s_1,\ldots,\FF_{n,d}^*s_d)-C(\FF_{n,1}s_1,\ldots,\FF_{n,d}s_d)\}
\]
and
$
  R_n^*(\s)  = R_{n,1}^*(\s) + R_{n,2}^*(\s)+ R_{n,3}^*(\s)+ R_{n,4}^*(\s)
$ with
\begin{eqnarray*}
R_{n,1}^*(\s) &= &
  \alpha_n^*(\FF_{n,1}^{*-}s_1,\ldots,\FF_{n,d}^{*-}s_d) - \alpha_n^*(\FF_{n,1}^{-}s_1,\ldots,\FF_{n,d}^{-}s_d)  \\
R_{n,2}^*(\s) &= &
  \alpha_n^*(\FF_{n,1}^{-}s_1,\ldots,\FF_{n,d}^{-}s_d) - \alpha_n^*(\s)\\
R_{n,3}^*(\s) &= &
 \alpha_n(\FF_{n,1}^{*-}s_1,\ldots,\FF_{n,d}^{*-}s_d) -
 \alpha_n(\FF_{n,1}^{-}s_1,\ldots,\FF_{n,d}^{-}s_d)   \\
 R_{n,4}^*(\s) &= &
 \sqrt{n} \left\{ C (\FF_{n,1}^{*-}s_1,\ldots,\FF_{n,d}^{*-}s_d) - C (\FF_{n,1}^{-}s_1,\ldots,\FF_{n,d}^{-}s_d) \right\}\\
 && + \sqrt{n}\left\{ C (\FF_{n,1}^{*}s_1,\ldots,\FF_{n,d}^{*}s_d) -  C (\FF_{n,1}s_1,\ldots,\FF_{n,d}s_d) \right\}.
\end{eqnarray*}
Let $\alpha_{n,j}^*(s) = \sqrt{n} ( \FF_{n,j}^* - \FF_{n,j})(s)$ and
 $\beta_{n,j}^*(s) = \sqrt{n} ( \FF_{n,j}^{-*} - \FF_{n,j}^{-})(s)$ be the bootstrap versions of the  empirical processes $\alpha_{n,j} (s)$ and $\beta_{n,j}(s)$, respectively.
Both converge to the same weak limit as
\[  \sup_{0\le s_j \le 1} | \beta_{n,j}^*(s_j)+\alpha_{n,j}^*(s_j)|= O (n^{-1/4} (\log n)^{1/2} (\log\log n)^{1/4})  \quad\text{ almost surely},\]
see displays (2.10') and (2.12') in Theorem 2.1 of Cs\"org\'o and Mason (1989).
It remains to show that
$\PP^* \{ L_n \sup_{\s} | R_{n}^*(\s) | > \eta \} \to0$
for all $\eta>0$, conditionally given all  sequences $(\X_1,\ldots,\X_n)\in \Omega_n$ for some sequence of events $\Omega_n\subset\RR^{d\times n}$ with $\lim_{n\to\infty} \PP(\Omega_n)=1$.\\

Let $\delta_n=n^{-1/4}$. (Other choices are possible as well.)
We have
\begin{eqnarray*}
\limsup_{n\to\infty} \PP^*\{ L_n\| R_{n,1}^*\|_\infty \ge \eta\}  \le \limsup_{n\to\infty}\PP^*\{ L_n \MM_n^*(\delta_n) \ge \eta\} +\limsup_{n\to\infty} \PP^*\{\max_j \| \beta_{n,j}^*\|_\infty \ge \sqrt{n} \delta_n\}=0,
\end{eqnarray*}
by Lemma  \ref{lemma:alpha_n^*}. Next, on the event $\max_j \| \beta_{n,j}\|_\infty \le \sqrt{n}\delta_n$ (that holds almost surely by the law of iterated logarithm),
\begin{eqnarray*}
\limsup_{n\to\infty} \PP^*\{ L_n\| R_{n,2}^*\|_\infty \ge \eta\} \le \limsup_{n\to\infty} \PP^*\{ L_n \MM_n^*(\delta_n) \ge \eta\} =0,
\end{eqnarray*}
by Lemma  \ref{lemma:alpha_n^*}.
On the event
$L_n\MM_n(\delta_n)  <\eta$ (that holds almost surely by Lemma  \ref{prop:oscillation}), we have
\begin{eqnarray*}
\limsup_{n\to\infty} \PP^*\{ L_n\| R^*_{n,3}\|_\infty \ge \eta\} \le \limsup_{n\to\infty} \PP^*\{  \max_j \| \beta_{n,j}^* \|_\infty > \sqrt{n}\delta_n\} =0
\end{eqnarray*}
by the weak convergence of $\beta_{n,j}^*$.
Finally, for some $s_j^*$ between $ \FF_{n,j}^{-*}(s_j)$ and $ \FF_{n,j}^{-}(s_j)$, and $s_j^{**}$ between $ \FF_{n,j}^{*}(s_j)$ and $ \FF_{n,j}(s_j)$, we have
\begin{eqnarray*}
 | R^*_{n,4} (\s) | &=&  \left|  \sum_{j=1}^d\{  C_j(s_j^*) \beta_{n,j}^*(s_j) + C_j(s_j^{**})  \alpha_{n,j}^*(s_j) \right| \\
 &\le&
  \sum_{j=1}^d | \beta_{n,j}^*(s_j) + \alpha_{n,j}^*(s_j) | + \sum_{j=1}^d |\alpha_{n,j}^*(s_j) | | C_j(s_j^*) - C_j(s_j^{**}) |.
\end{eqnarray*}
The first term is of order $O (n^{-1/4} (\log n)^{1/2} (\log\log n)^{1/4})$, uniformly in $s_j$. For the second term, we argue as in the proof of Proposition~\ref{prop:Z_n}.
First,
we observe that $| s_j^{**} - s_j^* | \le | s_j^*-s_j| + |s_j^{**}-s_j|  $ is of order $O_{p^*}(n^{-1/2})$. Second,
 since the class $\1\{ x\le t\} t^{-b} (1-t)^{-b}$ is a $P$-Donsker class for the uniform probability measure  $P$ on $[0,1]$, for all $0\le b<1/2$, see Van der Vaart and Wellner (1996), Example 2.11.15 (page 214),
  the weak convergence of the bootstrap empirical process [Van der Vaart and Wellner (1996, Theorem 3.6.1, page 347)] implies that
\[\sup_{0<s<1}|\alpha_{n,j}^*(s)| / (s^{b}(1-s)^{b} )= O_{p^*}(1). \]
Consequently, as in the proof of Proposition~\ref{prop:Z_n}, we find that, for some constant $K<\infty$,
\begin{eqnarray*}
 \sup_{\eps_n \le s_j\le 1-\eps_n }    |\alpha_{n,j}^*(s_j) | | C_j(s_j^*) - C_j(s_j^{**}) | &\le& K   | s_j^{**} - s_j^{*} |^r  s_j^{b-\beta} (1-s_j)^{b-\beta} \sup_{s_j}  |\alpha_{n,j}^*(s_j)| / (s^{b}(1-s)^{b} )
 \end{eqnarray*}
 which is of order $O_{p^*}(1) \cdot \max( n^{-r/2} \max(1, \eps_n^{b-\beta})$. On the other side,
 \begin{eqnarray*}
 \sup_{  s_j\not\in [\eps, 1-\eps_n] }    |\alpha_{n,j}^*(s_j) | | C_j(s_j^*) - C_j(s_j^{**}) | &\le& 2 \sup_{  s_j\not\in [\eps, 1-\eps_n] }    |\alpha_{n,j}^*(s_j) |\\&\le& 2\eps_n^b
  \sup_{s_j}  |\alpha_{n,j}^*(s_j)| / (s^{b}(1-s)^{b} ),
 \end{eqnarray*}
  which is of order $O_{p^*}(  \eps_n^{b})$. Combining both bounds yields $\sup_\s | R_{n,4}(\s)| = O_{p^*} ( \eps_n^b + n^{-r/2} \max(1, \eps_n^{b-\beta}))$.
 Taking $\eps_n=n^{-p}$ with $p$ depending on $b,\beta$ and $r$, we get that
 $\lim_{n\to\infty} \PP^*\{ L_n \sup_\s | R^*_{n,4}(\s) | \geq \eta \} = 0$  for all $\eta>0$, conditionally on all sequences  $(\X_1,\ldots,\X_n)\in \Omega_n$ for some sequence of events $\Omega_n$ with $\lim_n \PP(\Omega_n)=1$.
 This completes our proof.
\end{proof}

\subsection{\sc Proof of Theorem~\ref{Thm1}}
 By triangle inequality, we have,
\begin{eqnarray*}
\EE \left[ \sup_{h\in BL_1} \left| h (\ZZ_{n}) ] - \EE^* [ h (\ZZ_{n}^{\ast })] \right| \right]
& \le&   \sup_{h \in BL_1} \left| \EE[ h( \ZZ_n) - h(\widetilde \ZZ_n)] \right| \\ &&
+\EE \left[ \sup_{h \in BL_1} \left| \EE[ h(\widetilde \ZZ_n)] -\EE^* [ h(\widetilde \ZZ_n^*)] \right| \right]\\ &&+\EE\left[
\sup_{h \in BL_1} \left| \EE^* [ h(\widetilde \ZZ_n^*)-
 h(\ZZ_n^*)] \right|\right].
\end{eqnarray*}
In view of
Proposition \ref{prop:Z_n}
 and  Proposition \ref{prop:Z_n^*}, it remains to show that
the second term on the right is asymptotically negligible.
We recall  that
\begin{eqnarray*}
\widetilde{\ZZ}_{n}(f) &=& \sum_{k=1}^{2^d L_n}\sigma _{k}\widetilde{\ZZ}
_{n}(\s_{k})=  \sum_{k=1}^{2^d L_n}\sigma _{k}
 \int f_{k}(\x)\, d\alpha_n(\x),
\end{eqnarray*}
for
\[ f_k(\x)= \1\{ \x\le \s_{k} \} - \sum_{j=1}^d C_j(\s_k) \1\{ x_j\le s_{k,j}\} .\]
Now, let
$h_{f}(\x)= \sum_{k=1}^{2^d L_n}\sigma_{k}f_{k}(\x)$ so that
\begin{eqnarray}
\widetilde{\ZZ}_{n}(f)= \int h_f\, d\alpha_n,
\end{eqnarray}
and we can derive in the same way
\begin{eqnarray}
\widetilde{\ZZ}_{n}^*(f)= \int h_f\, d\alpha_n^*.
\end{eqnarray}
We now apply Theorem 3 in Radulovi\'c (2012), stated as Theorem
\ref{dragan2012} in the appendix for convenience. We need to
verify that
\begin{itemize}
\item
the $d+1$ classes
\begin{eqnarray*}
\G_k^{a} &=& \left\{ \1\{ \x\le \s_{k}\},\  \s_k\in [0,1]^d \right\},\\
\G_k^{(j)} &=& \left\{ C_j(\s_k) \1\{ x\le s_{k,j}\},\  \s_{k} \in [0,1]^d  \right\},\; j=1,\ldots,d,
\end{eqnarray*}
have VC--indices $V_{k}^{a}$ and $V_{k}^{(j)}$, respectively, with
$ \sum_{k=1}^{2^d L_n} ( V_{k}^{a} + \sum_{j=1}^d V_{k}^{(j)} ) \le K ( \log n)^\gamma$ for some finite constant $K$ and  some $0<\gamma<1$.
\item
 the class
$\mathcal{H}_{n}=\{h_{f}$ : $\ f\in \mathcal{F}_{n}\}$ has an envelope $H(\x)$ with $\EE[ H^4(\X)] <\infty$.
\end{itemize}

First we verify the VC property. The class $\mathcal{G}_{k}^{a}$
is   VC with VC-dimension  $V_{k}^{a}=d+1$ (Van der Vaart and
Wellner, 2000, page 135),
 while the class
$\G_k^{(j)}$ is a subclass of the class of functions $ c \1\{a\leq
x\leq b\} $ with $a,b\in \RR$ and $c>0$. This class has a VC index
$3$ : see van der Vaart and Wellner (2000), Problem 20, page 153.
Consequently
$$ \sum_{k=1}^{2^d L_n} ( V_{k}^{a} + \sum_{j=1}^d V_{k}^{(j)}) \le (4d+1) 2^d L_n \le K (\log n )^\gamma $$
for some $K<\infty$.\\

It remains to verify the envelope condition. We will show that $h_f(\textbf{x})$ has envelope $1+d+\sum_{j=1}^d TV(C_j)$.
Writing
\begin{eqnarray*}
g_{\x}(\s) &=& \1\{\x\leq \s\}-\sum_{j=1}^d C_{j}(\s)\1\{x_j\leq s_j\},
\end{eqnarray*}
we see that
\[
h_{f}(\x)=\sum_{k=1}^{L_n}c_{k}g_{\x}(B_k)
\]
for $c_k=\pm 1$ and the operation $\phi (B_k)$ defined in (\ref{Z_n:box}) for any function $\phi
:\RR^{d}\rightarrow \RR$.
Furthermore, writing
\begin{eqnarray*}
\gamma_{\x}(\s) =\1\{\x\leq \s\}, \;\; \zeta^{(j)}_{x}(\s) = C_{j}(\s)\1\{x\leq s_j\},\;j=1,\ldots,d,
\end{eqnarray*}
we have
 \begin{eqnarray*}
|h_{f}(\x)| &\leq& \sum_{k=1}^{L_n} |\gamma_{\x}(B_k)|+\sum_{j=1}^d \sum_{k=1}^{L_n}| \zeta_{x_j}^{(j)}(B_k)|\\
&\le& 1+ \sum_{j=1}^d \sum_{k=1}^{L_n}| \zeta_{x_j}^{(j)}(B_k)|
\end{eqnarray*}
since the boxes $B_1,\ldots,B_{L_n}$ are disjoint.
Since each $B_k$ is of the form $\prod_{j=1}^d (s_{k,j}^{1},s_{k,j}^{2}] $, there is a (fine enough) lattice partition $\Pi$ of $[0,1]^d$ with the property that each $B_k$ can be written as a union
of (disjoint) elements $A_{k_j}$, with $ A_{k_j}\in \Pi$. A little reflexion shows that, for each $1\le j\le d$,
\begin{eqnarray*}
 \sum_{k=1}^{L_n}| \zeta_{x_j}^{(j)}(B_k)| &\le& \sum_{A\in \Pi} | \zeta_{x_j}^{(j)}(A)|
\end{eqnarray*}
and, moreover, for $A_m= \prod_{j=1}^d (s_{m,j}^{1},s_{m,j}^{2}] \in \Pi$, $A_{m,-j}=\prod_{l\neq j} (s_{m,l}^1,s_{m,l}^2] $
and
$$C_j(\s_{-j} | t) := C_j(s_1,\ldots,s_{j-1},t,s_{j+1},\ldots,s_d),$$
 for every $\s_{-j}\in [0,1]^{d-1}$
and every $t\in [0,1]$,
 a little algebra gives the identity
\begin{eqnarray*}
\zeta_{x_j}^{(j)}(A_m) 
&=& \1\{x_j\leq s_{m,j}^{2}\} C_{j}(A_m) + \1\{s_{m,j}^{1}<x_j\leq
s_{m,j}^{2}\} C_j(A_{m,-j} | s_{m,j}^{1}).
\end{eqnarray*}
Since
\[ C_j(\s_{-j}|s_j)=\PP\{  \X_{-j} \leq \s_{-j} \, | \, X_j = s_j\},\]
we obtain
\begin{eqnarray*}
\sum_{k=1}^{L_n}| \zeta_{x}^{(j)}(B_k)| &\leq&
\sum_{A_m\in \Pi}  | \zeta_{x}^{(j)}(A_m)| \\
&\leq&
 \sum_{A_m\in \Pi}   |C_{j}(A_m)| + \sum_{A_m\in \Pi} \1\{s_{m,j}^{1}<x_j \leq
s_{m,j}^{2}\} \PP\{ \X_{-j} \in A_{m,-j} \, | \, X_j = s_{m,j}^1\}\\
&\leq&
\text{TV}(C_j)  +   \sum_{A_m\in \Pi} \1\{s_{m,j}^{1}<x_j \leq
s_{m,j}^{2}\} \PP\{ \X_{-j} \in A_{m,-j} \, | \, X_j = s_{m,j}^1\}.
\end{eqnarray*}
Let $A_{m(\x)}\in\Pi$ with $\x\in A_{m(\x)}$ and $s_j^{1}<x\le s_{j}^{2}$ with $(s_{j}^{1},s_{j}^{2}]$ be the projection of $A_{m(\x)}$ on the j-th axis of the lattice.
Then, the last term on the right of the previous display can be bounded as follows:
\begin{eqnarray*}
&& \sum_{A_m\in \Pi} \1\{s_{m,j}^{1}<x_j \leq s_{m,j}^{2}\} \PP\{ \X_{-j} \in A_{m,-j} \, | \, X_j = s_{m,j}^1\}\\
&& \le  \sum_{A_m\in \Pi,\ s_{m,j}^1=s_j^1, \ s_{m,j}^2 =s_j^2 }  \PP\{ \X_{-j} \in A_{m,-j} \, | \, X_j = s_{j}^1\}\\ &&
\le1
\end{eqnarray*}
since the boxes $A_{m}\in \Pi$ and therefore $A_{m,-j}$ are disjoint. We have shown that
the class $\mathcal H_n$ has envelope $1+d+\sum_{j=1}^d TV(C_j)$.\\
We can now apply Theorem \ref{dragan2012} to conclude that
$$\lim_{n\to\infty}  \EE \left[ \sup_{h \in BL_1} \left| \EE[ h(\widetilde \ZZ_n)] -\EE^* [ h(\widetilde \ZZ_n^*)] \right| \right]  = 0,$$ and the proof is complete.
\qed

\subsection{\sc Proof of Theorem~\ref{CoreThParam}}

We proceed as in the proof of Theorem~\ref{Thm1}. We write
$\widehat C= C_{\widehat\theta}$ and $\wh C^*= C_{\wh \theta^*}$.
Recall that
\[ \YY_n= \ZZ_n - \sqrt{n}(\wh C-C).\]
We may replace $\ZZ_n$ by $\widetilde\ZZ_n$ with impunity since
\begin{eqnarray*}
&&\sup_{h\in BL_1}\left| \EE [ h( \YY_n) - h(\widetilde\ZZ_n-\sqrt{n}(\wh C-C) )] \right| \\
&\le& \delta + 2\PP\left\{   \sup_{f\in \F_n} \left| \YY_n(f)- \widetilde \ZZ_n(f) + \sqrt{n} (\wh C-C)(f) \right| \ge\delta \right\}\\
&=&\delta +2 \PP\left\{ \sup_{f\in \F_n } \left| \ZZ_n(f)- \widetilde \ZZ_n(f) \right|\ge\delta \right\} \\
&\to&\delta \ \text{as } n\to\infty,
\end{eqnarray*}
for every $\delta>0$,
as in the proof of  Proposition~\ref{prop:Z_n}. Next, by the mean value theorem and
assumptions (C3) and (C4), we have
\begin{eqnarray*}
\sqrt{n}(\wh C-C)(\s) &=&
\sqrt{n} (\wh\theta-\theta_0)' \dot{C}_{\theta_0}(\s)+ \sqrt{n}(\wh\theta-\theta_0)'\{\dot{C}_{\tilde \theta}(\s)-\dot{C}_{\theta_0}(\s)\}
\\
&& \text{ for some $ \tilde \theta$ between $\wh\theta$ and $\theta_0$}
\\
&=&
 \left( \int \psi\, d\alpha_n   +n^{1/2} \eps_n \right)' \dot{C}_{\theta_0}(\s) + \sqrt{n} (\wh\theta-\theta_0)'\{ \dot{C}_{\tilde \theta}(\s)-\dot{C}_{\theta_0}(\s)\}
\\
&=&\left(\int \psi\, d\alpha_n\right)' \dot{C}_{\theta_0} (\s)  + R_n(\s)
\end{eqnarray*}
for some remainder term $R_n$ that satisfies
\begin{eqnarray*}
| R_n(\s)  | &\le& n^{1/2}\|\eps_n\|_2 \| \dot{C}_{\theta_0}(\s)\|  _2 +  K n^{1/2}\|\wh\theta -\theta_0\|_2^{1+\nu}\\
&=& O_p( n^{1/2}\|\eps_n \|_2+ n^{-\nu/2} )\\
&=& o_p(1/L_n).
\end{eqnarray*}
This bound holds uniformly in $\s$.
Consequently,  for $$\widetilde \YY_n(f)=\sum_{k=1}^{2^d L} \sigma_k \widetilde \YY_n(\s_k)$$ based on
\begin{eqnarray*}
\widetilde \YY_n(\s) &=& \widetilde\ZZ_n(\s) -  \left(\int \psi
\, d \alpha_n\right)' \dot{C}_{\theta_0}(\s),
\end{eqnarray*}we have
\begin{eqnarray*}
\sup_{h \in BL_1} \left| \EE [ h( \widetilde \ZZ_n -\sqrt{n}(\wh C-C) )] - \EE [ h( \widetilde \YY_n)] \right|
&=& \sup_{h \in BL_1} \left| \EE [ h( \widetilde \YY_n- R_n) ] - \EE[ h  ( \widetilde \YY_n)] \right| .
\end{eqnarray*}
Since
$$\sup_f | R_n(f) |\le 2^dL_n \sup_{\s} |R_n(\s)| \to 0
$$ in probability, we get $\sup_h | \EE[ h( \widetilde \ZZ_n -\sqrt{n}(\wh C-C))]- \EE[ h(  \widetilde \YY_n)] | \to0$, as $n\to\infty$. We conclude that
\begin{eqnarray*}
\limsup_{n\to\infty}
\sup_{h\in BL_1}\left| \EE[ h(\YY_n) ]- \EE[ h( \widetilde \YY_n)] \right| =0.
\end{eqnarray*}

For the bootstrap counterpart, we can argue in the same way. Using the expansion
\begin{eqnarray*}
\sqrt{n}(\wh C^*-\wh C)(\s) &=& \left( \int \psi\, d\alpha^*_n \right)' \dot{C}_{\theta_0} (\s)  + R_n^*(\s)
\end{eqnarray*}
for some remainder term $R_n^*$ that satisfies
\begin{eqnarray*}
\sup_{\s} | R_n^*  (\s) | &\le& K_0n^{1/2}\|\eps_n^*\| _2  + K_1
n^{1/2}\|\wh\theta -\theta_0\|_2^{1+\nu}+ K_2
n^{1/2}\|\wh\theta^*-\wh\theta\|_2^{1+\nu},
\end{eqnarray*}
for some finite constants $K_0, K_1$ and $K_2$. We check that the
processes $\YY_n^*$ and $\widetilde \YY_n^*$ are close with $\widetilde \YY_n^*$  based on
\begin{eqnarray*}
\widetilde \YY_n^*(\s)&=& \widetilde\ZZ_n^*(\s)- \left( \int \psi \, d \alpha_n^*\right)'  \dot{ C}_{\theta_0}(\s).
\end{eqnarray*}

Note that $\widetilde
\YY_n(f)= \sum_{k} \sigma_k \widetilde \YY_n(\s_k) =\int \left( \sum_{k} \sigma_k  g_k \right) \, d\alpha_n $  with
\begin{eqnarray*}
 g_k(\x) &=& \1\{ \x\le \s_k\} - \sum_{j=1}^d C_j(\s_k) \1\{ x\le s_{k,j}\} - \left(  \psi (\x) \right)'  \dot{C}_{\theta_0}(\s_k).
\end{eqnarray*}
As in the proof of Theorem \ref{Thm1}, it remains to verify the two conditions of Theorem \ref{dragan2012}.
Since the only difference with the proof of Theorem \ref{Thm1} is the addition of the term  $\left( \psi (\x)\right) ' \dot{C}_{\theta_0}(\s_k) $, we concentrate on the class
 of functions $ \left( \psi (\x) \right)' \dot{C}_{\theta_0}(\s_k)$. Since it is a
 subclass of $c' \psi(\x)$ with $c\in\RR^p$, its VC dimension trivially is equal to $p$.
 Moreover, it is not hard to see from the proof of Theorem \ref{Thm1} that
 \begin{eqnarray*}
\left|\sum_{k=1}^{2^d L_n}\sigma _{k}g_{k}(\x)\right| &\leq& 1+d+\sum_{j=1}^d  TV(C_j) + \|\psi(\x) \|TV(\dot{C}_{\theta_0}).
\end{eqnarray*}

Since $\EE [\|\psi(\X)\|_2^4]<\infty$, the conditions of Theorem \ref{dragan2012} are met, and we conclude that
$$ \EE \left[ \sup_{h\in BL_1} \left| \EE[ h( \widetilde\YY_n) ] -\EE^*[ h( \widetilde \YY_n^*)]\right|\right] \to0$$  as $n\to\infty$.
\qed

\subsection{\sc Proof of Proposition~\ref{ThC3C3bis}}

From the proofs of Proposition \ref{prop:Z_n} and Proposition \ref{prop:Z_n^*}, we   see that
$$ \sup_{\u\in [0,1]^d} | \ZZ_n(\u) -\widetilde \ZZ_n (\u)  | = O_p\left(n^{-\mu}\right)\;\;\text{and}$$
$$ \sup_{\u\in [0,1]^d} | \ZZ_n^*(\u) -\widetilde \ZZ_n^* (\u)  | = O_{p^*}\left(n^{-\mu}
\right),$$ almost surely, for some $\mu >0$. The result follows
 after  integration by parts. \qed

\subsection{\sc Proof of Corollary~\ref{CorNewBoots}}
By the delta-method,
$\{\widetilde \YY_n(\s), \s \in [0,1]^d\}$ converges  towards a Gaussian process in $\ell^{\infty}([0,1]^d)$.
The proof of Theorem~\ref{CoreThParam} shows that
$\limsup_{n\to\infty}
\sup_{h\in BL_1}| \EE[ h(\YY_n)- h( \widetilde \YY_n)] |=0$.
Hence,  the process $\YY_n$ converges weakly to the same weak limit as $\widetilde \YY_n$.
This proves the first claim.
The second part of the Corollary is a straightforward consequence of Theorem~\ref{CoreThParam} and the triangle inequality. \qed

\appendix

\section{}
\setcounter{equation}{0}

Let $X_1,\ldots,X_n$ be independent random variables with probability measure $P$. Let $\PP_n$ be the empirical probability measure, putting mass $1/n$ at each observation, and let $\PP_n^*$ be the nonparametric bootstrap measure based on $n$ independent observations from $\PP_n$. We index the empirical process $\sqrt{n}(\PP_n-P)$ and its bootstrap counterpart $\sqrt{n}(\PP_n^*-\PP_n) $  by functions $f$ that belong to a sequence of classes $\F_n$.\\

\begin{thm}
\label{dragan2012}
Let $d_n$ be an integer sequence and, for each $1\le i\le d_n$,  let
$\mathcal{G}_{i,n}$  be
a VC class of functions with VC index $V_{i,n}$ and
$$\sum_{i=1}^{d_{n}}V_{i,n}\le K (\log n)^{\gamma },$$
for some $K<\infty$ and $0<\gamma <1$.
Set
$$\mathcal{F}_{n}=\left
\{f=\sum_{i=1}^{d_{n}}g_{i}:\ g_i\in \mathcal{G}_{i,n}\right\},$$
and suppose that
 there exists an envelope function $F\geq \sup_{f\in \F_n} |f|$, independent of $n$, with $\EE[F^4(X)]<\infty$.
Then,
\[
\limsup_{n\to\infty} \EE \left[ \sup_{h\in BL_1} \left| \EE[h (
\sqrt{n}(\PP_n-P)) ] - \EE^* [ h(
 \sqrt{n}(\PP_n^*-\PP_n))] \right| \right] =0
 .\]
\end{thm}
\begin{proof}
See Theorem 3 in Radulovi\'c (2012).
\end{proof}
\bigskip

\section{}
\setcounter{equation}{0}

Set $\MM_n(\delta)$ as in (\ref{M_n}) for $\delta\ge 0$, and define
\[ \psi(x)= 2x^{-2} \{ (1+x) \log (1+x) -x \}, \qquad x\in (-1,0)\cup(0,\infty)\]
and $\psi(-1)=2$ and $\psi(0)=1$. This function is continuous and decreasing.\\

\begin{prop}\label{segers}
There exist constants $K_1$ and $K_2$ such that
\begin{eqnarray}
\PP\left\{ \MM_n(a) \ge \lambda \right\} \le \frac{K_1}{a} \exp\left\{ - \frac{ K_2 \lambda^2}{a} \psi\left( \frac{\lambda}{\sqrt{n} a } \right) \right\}
\end{eqnarray}
for all $ a\in (0,1/2]$ and all $\lambda\in [0,\infty)$.
\end{prop}
\begin{proof}
See Proposition A.1 of  Segers (2012). 
\end{proof}

\section{}
\label{computation}

We present a stochastic
optimization algorithm  that approximates  $\widetilde{\TT}_n$. 
 The algorithm  is  based on Pure Random Search and easily   implementable.\\

\begin{enumerate}
\item[{\sc Step} 1.]
Compute and store, for all $i_{j}\in \{0,...,\lfloor n^{1/d} \rfloor\},$
$$F(i_{1},...,i_{d}):= \ZZ_{n}\left(\frac{i_{1}}{ n^{1/d} }%
,...,\frac{i_{d}}{n^{1/d}}\right)$$

\item[{\sc Step} 2.]
\begin{enumerate}
\item
Compute and store,  for all $B_{i}=\Pi _{j=1}^{d}( {a_{i,j}}{n^{-1/d}}, {b_{i,j}}{%
n^{-1/d}}],$  with $%
a_{i,j},b_{i,j}\in \{0,...,\lfloor n^{1/d}\rfloor \}$ and $a_{i,j}<b_{i,j}$,
$$G(B_{i}):=\Delta _{a_{i,1},b_{i,1}}^{1}\Delta
_{a_{i,2},b_{i,2}}^{2}...\Delta _{a_{i,d},b_{i,d}}^{d}F.$$
\item
Rank the  
 $B_{i} $ according to  $G(B_{1})\geq ...\geq G(B_{m})$.
We suggest    $m=n$ as the default value.\\

\end{enumerate}
\item[{\sc Step} 3.] 
\begin{enumerate}
\item
Sample without replacement $(A_{1},\ldots, A_{L_n})\in\B=
\{B_{1},...,B_{m}\}$. 
\item
Compute, for $\mathbf{A}=(A_{1},...,A_{L_n})$ of part 3(a),
\begin{eqnarray*}
\T(\mathbf{A}) 
&=&G(A_{1})+G(A_{2})\1_{\{ A_{1}\cap
A_{2}=\emptyset\} }+\ldots+G(A_{L_n})\1_{\{ A_{1}\cap \ldots \cap A_{L_n}=\emptyset \}}.
\end{eqnarray*}

\item 
Repeat  parts (3a) and (3b)  $K $ times. We suggest $K=10^4$ as the default value.    \\    
\end{enumerate}
\item[{\sc Step} 4.]
Find $\T(\mathbf{A}^{o}) = \max_{\mathbf{A}}  \T(\mathbf{A})$ with the maximum taken over
the obtained list $\mathbf{A}_1,\ldots,\mathbf{A}_K$ in step (3),
and use this to approximate $\widetilde \TT_n $.

\end{enumerate}
 \medskip

{\sc Remark.}
(Computational cost): Step 1 requires $n$ computations. We would like to caution that Step 1, although negligible if coded in C++ or
Fortran, tends to be very slow if performed using more elaborate programming
languages like R or Mathematica.
 Step 2 requires less than $n^2$ 
summations. 
Step 3, the verification whether
 $L_n$ rectangles overlap,  requires at most $L_n^{2}/2$
 verifications, each in turn  requiring   $2d$ operations.
Thus,  we need at most  $d K L_n^{2}  $  operations  in Step 3. 
 
For a typical (larger) case $n=800$, $\ d=2,$ and
$L_n=4$, the number of computations needed for Step 2 and Step 3 is bounded by $
10^{6}.$ 
Since an ATV test typically requires $10^{3}$ bootstrap samples,
the total number of summations needed is of the order $10^{9}$. A typical
desktop computer (using C++ or Fortran code) needed  less than 5 seconds.

\mds

{\sc Remark.} 
(Improvements):
We took $m=n$ and $K=10^4$. Smaller values for $m$ and $K$ would speed up the computation, while larger values would offer more guarantees that we find the true optimum. We experimented with $m=10n$,  $m=100n$, $K=10^5$  and $K=10^6$, but we did not observe any significant improvements.\\
 
\mds

It is possible to enhance the proposed
algorithm by including an additional step, which would concentrate on local search.
Implementation of more sophisticated algorithms such as the  Accelerated Random Search
algorithm (Appel et al., 2004) would allow
us to quickly search the the neighborhod of $\mathbf{A}^{o}$.
We experimented with this
approach, and although it produced slightly larger values for statistic $%
\widetilde{{\mathbb{T}}}_{n}$, the overall performance did not significantly
change.  
We suspect that such an additional step would be  more valuable in  dimensions   $d>2$.
For a good review of optimization schemes relevant to this
scenario we refer to the paper by Hvattuma and Gloverb (2009), where the
authors describe eight optimization schemes and contrasts their performance
on numerous test functions in higher dimensions (up to dimension 64).


\end{document}